\documentclass[12pt]{article}
\usepackage[utf8]{inputenc}
\usepackage{indentfirst}
\usepackage{amssymb}
\usepackage{verbatim}
\usepackage{amsthm}
\usepackage{fleqn}
\usepackage{graphicx}
\usepackage{epstopdf}
\usepackage{amsfonts}
\usepackage[numbers]{natbib}
\usepackage[colorlinks,citecolor=blue,urlcolor=blue]{hyperref}
\usepackage{amsmath}
\usepackage{tikz}
\usepackage{bm}

\usepackage{geometry}
\usepackage{pifont}
\newtheorem{theorem}{Theorem}[section]
\newtheorem{lemma}[theorem]{Lemma}

\newtheorem{remark}[theorem]{Remark}

\theoremstyle{plain}

\newcommand{\cF}{\mathcal{F}}
\newcommand{\cG}{\mathcal{G}}
\newcommand{\cH}{\mathcal{H}}

\newcommand{\cU}{\mathcal{U}}

\newcommand{\cN}{\mathcal{N}}
\newcommand{\HH}{\mathbb{H}}
\newcommand{\FF}{\mathbb{F}}
\newcommand{\RR}{\mathbb{R}}

\newcommand{\LL}{\mathbb{L}}
\newcommand{\cO}{\mathcal{O}}

\numberwithin{equation}{section}

\begin{document}
	\title{Linear-quadratic control  for mean-field  backward stochastic differential  equations with random  coefficients  }
	
	\author{Jie Xiong%
	\thanks{Department of Mathematics and SUSTech International center for Mathematics, Southern University
		of Science and Technology, Shenzhen, 518055, P. R. China. Email: xiongj@sustech.edu.cn.}
	\and Wen Xu%
	\thanks{Department of Mathematics, Southern university of Science and Technology, Shenzhen, 518055, P. R. China. Email:12231279@mail.sustech.edu.cn}
	\and Ying Yang%
	\thanks{Department of Mathematics, Southern University
		of Science and Technology, Shenzhen, 518055, P. R. China. Email:12331007@mail.sustech.edu.cn}
}
	\date{}
	\maketitle
	\begin{abstract}
		In this paper, we study the linear-quadratic control problem for mean-field  backward stochastic differential  equations (MF-BSDE) with random  coefficients. We first derive a preliminary stochastic maximum principle to analyze the unique solvability of the optimality system for this control problem through the variational method. Subsequently, we reformulate the mean-field linear-quadratic (MF-BSLQ) problem as a constrained BSDE control problem by imposing constraints on the expectation processes, which we solve using the Extended Lagrange multiplier method. Finally, we derive an explicit expression for the optimal control associated with Problem (MF-BSLQ).
	\end{abstract}
	
	\text{\bf Keywords}: Extended LaGrange multiplier method, mean-field control, linear quadratic control problem, random coefficients, Riccati equation, stochastic maximum principle.
	
	\text{\bf AMS Subject Classification}: 49N10, 60H10, 93E20.

	\section{ Introduction}
	
	Let $(\Omega, \mathcal{F}, \mathbb{F}, \mathbb{P})$ be a complete filtered probability space on which an one dimensional standard Brownian motion $\{W(t):0\leq t\leq +\infty\}$ is defined, where $\mathbb{F}=\{\mathcal{F}_t\}_{t\geq 0}$ is  the natural filtration generated by $W(t)$ and augmented by all  $\mathbb{P} $ null sets.  We consider the following   controlled linear mean-field backward stochastic differential equation (MF-BSDE for short)  with random coefficients: 
	\begin{equation}\label{state1}
		\left\{\begin{aligned}
			d Y(s)= &- \{A(s) Y(s)+\bar{A}(s) \mathbb{E}[Y(s)]+B(s) u(s)+\bar{B}(s) \mathbb{E}[u(s)] \\
			& +C(s) Z(s)+\bar{C}(s) \mathbb{E}[Z(s)]\} d s+Z(s) d W(s), \quad s \in[0, T], \\
			Y(T)= & \xi,
		\end{aligned}\right.
	\end{equation}
	where $A(\cdot), \bar A(\cdot), C(\cdot), \bar C(\cdot):[0,T]\times \Omega \to \mathbb{R}^{n\times n}$ and $ B(\cdot), \bar B(\cdot) :[0,T]\times \Omega \to  \mathbb{R}^{n\times m}$,  are matrix-valued $\mathbb{F}$-progressively measurable processes, and  $u(\cdot)$ is an $\FF $-progressively measurable process satisfying 
	$\mathbb{E}\int_0^T|u(s)|^2ds < \infty$.   The terminal state $\xi$ is an $\mathcal{F}_T$-measurable random vector, which is square-integrable and will be fixed throughout this article.  In equation (\ref{state1}), $\big(Y(\cdot), Z(\cdot)\big)$ valued in $\mathbb{R}^{n} \times \mathbb{R}^{n} $  is the state process, and $u(\cdot)$ valued in $\mathbb{R}^m$  is the control process.

	Now we introduce the following cost functional with random coefficients: 
	\begin{equation}\label{cost}
		\begin{aligned}
			J(u(\cdot)) \triangleq \mathbb{E}\Big\{ & \langle G Y(0), Y(0)\rangle 
			+\int_0^T[\langle Q(s) Y(s), Y(s)\rangle+\langle\bar{Q}(s) \mathbb{E}[Y(s)], \mathbb{E}[Y(s)]\rangle \\
			& +\langle R(s) Z(s), Z(s)\rangle+\langle\bar{R}(s) \mathbb{E}[Z(s)], \mathbb{E}[Z(s)]\rangle 
			+\langle N(s) u(s), u(s) \rangle  \\
			& + \langle \bar N(s)  \mathbb{E} [u(s)], \mathbb{E} [u(s)]\rangle d s\Big\},
		\end{aligned}
	\end{equation}
	where $G$  is an $\mathcal{F}_t$ measurable random matrix, $Q(\cdot), \bar{Q}(\cdot), R(\cdot), \bar{R}(\cdot) : [0,T]\times \Omega \to \mathbb{R}^{n\times n}$, and  $ N(\cdot),  \bar{N}(\cdot) : [0,T]\times \Omega \to  \mathbb{R}^{n\times m} $. We assume that the control $u(\cdot)$ belonging to the following space
	$$
	\mathcal{U}[0, T] =  \left\{u(\cdot):[0, T] \times \Omega \rightarrow \mathbb{R}^m \mid u(\cdot)  \text { is $\cF_t$-adapted  and } \mathbb{E} \int_0^T|u(t)|^2 d t<\infty\right\},
	$$
	with the state equation (\ref{state1}) and the cost founctional (\ref{cost}),
	our mean-field backward stochastic linear quadratic  (MF-BSLQ  for short)  optimal control problem with random coefficients can be stated as follows.

	\textbf{Problem (MF-BSLQ).}  For any given $t \in[0, T)$ and a square-integrable $\mathcal{F}_T$-measurable terminal state  $\xi $, find a $u^*(\cdot) \in \mathcal{U}[0, T]$ such that
	$$
	J\left(u^*(\cdot)\right)=\inf _{u(\cdot) \in \mathcal{U}[0, T]} J(u(\cdot)).
	$$
	
	Unlike forward  stochastic differential equations (SDEs for short), which have a single   solution,  the solution to  backward stochastic differential equations   (BSDEs for short)  is a pair of   processes \((Y(\cdot), Z(\cdot))\),  where \(Z(\cdot)\) is necessary for the BSDEs to have   solutions and can be interpreted as a risk-adjustment factor.  Linear BSDEs (without mean-field terms)  were first introduced as an adjoint equation in the Pontryagin maximum principle by Bismut \cite{bismut1978introductory} in 1973. Then,  the theory of nonlinear   BSDEs were initially  developed by Pardoux and Peng \cite{pardoux1990 } in 1990.  Since then, 
	applications of BSDEs have been explored in several areas, including  mathematical finance and stochastic control.  For instance,  Hu \textit{et al.} \cite{hu2023constrained} used  BSDEs to provide semiclosed optimal strategies and optimal values for both the monotone mean-variance problem and the classical mean-variance problem with convex cone trading constraints in a market with random coefficients.   Lim and Zhou \cite{lim2001linear} were the first to study the linear-quadratic (LQ) problem for linear BSDEs with quadratic cost criteria, using the completion-of-squares technique to obtain the optimal solution. Moreover, the controlled BSDEs have numerous important applications, see Peng \cite{peng1993backward}, Hamadene \textit{et al.} \cite{hamadene1997bsdes}, Pham \cite{pham2009continuous} and all the references therein.

	However, most of the existing research primarily focuses on BSDEs with deterministic coefficients. Sun and Wang \cite{sun2021linear} were the first to study the control problem of linear BSDEs with random coefficients. They introduced a novel stochastic Riccati-type equation and decoupled the optimality system to solve the control problem explicitly. We  highlight that the work \cite{sun2021linear}  on the linear  BSDE  control problem with random coefficients plays a significant role in our current study.

	Following the development of BSDEs,  mean-field backward stochastic differential equations (MF-BSDEs for short)  were later proposed.   In recent years, a large amount of literature has studied this issue.  For instance, the work done by Buckdahn  \textit{et al.}    \cite{buckdahn2009mean} and  Buckdahn \textit{et al.}   \cite{buckdahn2009mean1}. The former work focuses on the properties of an approximation to the solution of a mean-field BSDE, while the latter provides a probabilistic interpretation of semilinear McKean-Vlasov partial differential equations (PDEs) using MF-BSDEs. Moreover, Carmona and Delarue \cite{carmona2015forward} provided a detailed probabilistic analysis of optimal control for nonlinear stochastic dynamical systems of the McKean-Vlasov type. Hao  \textit{et al.}   \cite{hao2022solvability} studied multi-dimensional MF-BSDEs with quadratic growth.  Chen \textit{et al.}   \cite{chen2020p} studied the unique  \(L^p\)-solution  of  MF-BSDEs.    As for the applications of MF-BSDEs, interesting reader can    refer to   \cite{du2019linear}, 
	\cite{huang2019asymmetric}, \cite{li2019linear} and all the references therein. It is worth noting that the works mentioned so far focus on MF-BSDEs with deterministic coefficients and no existing literature deals with the controlled MF-BSDEs problem with random coefficients.

	In the present paper, we focus on the mean-field LQ control problem for BSDEs with random coefficients. As detailed in Xiong and Xu \cite{xiong2024mean}, a primary challenge in addressing such problems arises from the presence of terms like $\mathbb{E}[\bar{A}(\cdot)X(\cdot)]$ in the adjoint equation. The inability to separate these terms into $\mathbb{E}[\bar{A}(\cdot)]\mathbb{E}[X(\cdot)]$ prevents the decoupling of the original optimal system (\ref{eq:SMP0}). To overcome this issue, we apply the Extended Lagrange Multiplier (ELM) method outlined in \cite{xiong2024mean}, which allows us to decompose the Problem (MF-BSLQ) into two more manageable subproblems free from mean-field terms.
	Our first step is to transform the Problem (MF-BSLQ) into a constrained BSLQ control problem with random coefficients. This transformation facilitates deriving the expression for the optimal control in Theorem \ref{th:uast}, which ultimately depends on the solution of a system of Riccati equations and a backward stochastic differential equation (BSDE). This dependency also implies the existence and uniqueness of the solution to the Riccati equation. The second step is to solve the general constrained control problem using the ELM method described above.

	The key contributions of our study are as follows. First, we establish the existence of a unique optimal control for Problem (MF-BSLQ) and provide its characterization using the stochastic maximum principle. Next, we reformulate Problem (MF-BSLQ) as a constrained BSDE problem by constraining the expectation processes to deterministic functions and solve it using the ELM method. Finally, we derive an explicit expression for the optimal control of Problem (MF-BSLQ), providing a comprehensive and systematic solution to the mean-field LQ control problem for BSDEs with random coefficients.
	
	The remainder of this paper is organized as follows. In Section \ref{sec2}, we introduce the necessary notations and related spaces, along with some preliminary results on BSDEs and mean-field SDEs. In Section \ref{sec3}, we reformulate Problem (MF-BSLQ) into Problems (a) and (b), and present the main results of the paper, including Theorems \ref{preSMP} and \ref{main}. Section \ref{sec4} provides a detailed proof of Theorem \ref{preSMP}. In Section \ref{sec5}, we describe the process of solving Problem (a). Finally, in Section \ref{sec6}, we provide a proof of Theorem \ref{main}.


	\section{ Preliminaries}\label{sec2}
	In this section, we introduce the necessary notations and related spaces, and then present the key assumptions of this article. Moreover,  we state the existence and uniqueness of the solution to the state equation (\ref{state1}) and the adjoint equation  (\ref{eq:sec3-X}).
	
	For a random variable $\xi$,  we write $\xi \in \mathcal{F}_t$ if $\xi$ is  $\mathcal{F}_t$ measurable, and for a stochastic process $\phi(\cdot)$,  we write $\phi(\cdot) \in \mathbb{F}$ meaning that $\phi(\cdot) $   is progressively measurable.	  For Euclidean space $\mathbb{H} = \mathbb{R}^n,  \mathbb{R}^{m\times n},   \mathbb{S}^n_+$,  and $p,q > 0$, we introduce the following spaces.  
	\begin{itemize}
		\item$L^{p,q}_\FF(\HH)\equiv L^p_\FF(\Omega;L^q([0,T];\HH))$: the space of \(\cF_t\)-  measurable processes $$X:[0,T]\times\Omega\rightarrow\HH\text{ with }\mathbb{E}\Big(\int_0^T ||X(s,\omega)||_\HH^qds\Big)^{p}<\infty.$$ 
		
		Especially, we denote \(L^2_\FF(\HH)\equiv L^{1,2}_\FF(\HH)\).
		
		\item $L_{\mathbb{F}}^{2,c}\big(\Omega; C([0,T]; \mathbb{H}) \big)$: the space of continuous $\cF_t$-   measurable processes
		$$X: [0, T]\times \Omega \to \mathbb{H}\;  \text{with} \;  \mathbb{E}\left[\text{sup}_{0 \leq s \leq T}  \|X(s,\omega)\|_\HH ^2\right]< \infty.$$ 
		
		\item $L^\infty_{\mathbb{F}}(\mathbb{H})\equiv L_{\mathbb{F}}^{\infty}(0, T;  \mathbb{H})$: the space of  $\cF_t$-  $  \mathbb{H}$  valued  bounded processes. 
		
		\item $L^{\infty,c}_{\FF}(\HH)\equiv L^{\infty,c}_\FF (0,T;\HH):$ the space of \(\cF_t\)-  \(\HH\)-valued bounded continuous pprocesses.
		
		\item $L_{\cG}^{2}(\HH)\equiv L^2_\cG(\Omega; \mathbb{H})$: the space of $\cG$-measurable $\mathbb{H}$-valued square integrable  random variables, where \(\cG\subset \FF\) is a sub-\(\sigma\)-field of \(\FF\).
		
		\item $L^\infty_\cG(\HH)\equiv L^\infty_\cG(\Omega;\HH):$ the space of \(\cG\)-measurable \(\HH\)-valued bounded random variables.

		\item $\LL^2\equiv L^2(0, T; \RR^n)$: the space of deterministic square integrable  functions $\alpha: [0, T] \to \mathbb{R}^n$ with $\int^T_0 \alpha (s)^2ds < \infty.$

	\end{itemize}	
	Throughout this article, we  impose the following assumptions:
	\begin{itemize}
		\item[$\mathbf{(H1):}$]
		$ A(\cdot), \bar A(\cdot), C(\cdot), 
		\bar C(\cdot) \in L_{\mathbb{F}}^{\infty}(\mathbb{R}^{n\times n}),$ and 
		$B(\cdot), \bar B(\cdot) \in  L_{\mathbb{F}}^{\infty}(\mathbb{R}^{n\times m})$. 
		
		\item [$\mathbf{(H2)}:$] $Q(\cdot),\bar  Q(\cdot), R(\cdot), \bar{R} (\cdot) \in L_{\mathbb{F}}^{\infty}( \mathbb{S}_{+}^n),$   and $  N(\cdot), \bar N(\cdot) \in L_{\mathbb{F}}^{\infty}\big(\mathbb{S}_{+}^m),$ $G  \in L_{\mathcal{F}_0}^{\infty}(\mathbb{S}_{+}^n).  $  Moreover, there exists a constant $\delta > 0$ such that
		$$
		N(s) \geq \delta I_{m}, \;  R (s)\geq \delta I_{n}, \; a.e.\; s\in [0, T], a.s.,
		$$
		where $I_m$ is the $m\times m$ identity matrix and $I_n$ is the \(n\times n\) identity matrix.
	\end{itemize}
	
	For the well-posedness of equations (\ref{state1}) and (\ref{eq:sec3-X}), we present the following theorems, the proofs  are similar to those in Xiong and Xu \cite{xiong2024mean} (Theorem 3.2 and Theorem 3.1, respectively),  so we omit the details  here. 
	\begin{theorem}\label{thm31}
		Let  (H1) hold.   Then, for any terminal state $ \xi \in L^2_{\mathcal{F}_T}(\mathbb{R}^n)$ and control $u(\cdot)\in \mathcal{U}[0,T]$, the state equation $(\ref{state1})$ admits a unique   solution $$\big(Y(\cdot),Z(\cdot)\big)\in L^{2,c}_{\mathbb{F}}(\mathbb{R}^n)\times L^2_{\mathbb{F}}(\mathbb{R}^n).$$
		Moreover, there exists a constant $ K >0$, which is independent of $ \xi$ and $u(\cdot)$ such that 
		\begin{equation}
			\mathbb{E}\Big[\sup_{0\leq s\leq T} |Y(s)|^2+\int_0^T |Z(s)|^2ds\Big]\leq K\mathbb{E}\Big[|\xi|^2+\int_0^T |u(s)|^2ds\Big].
		\end{equation}  
	\end{theorem}
	
	Next,  given a pair $\big(Y^\ast(\cdot),Z^\ast(\cdot)\big)\in L^{2,c}_{\mathbb{F}}(\mathbb{R}^n)\times L^2_{\mathbb{F}}(\mathbb{R}^n) $, we consider     the following mean-field SDE with random coefficients:\ for $ s \in[0, T], $
	\begin{equation}\label{eq:sec3-X}
		\left\{
		\begin{aligned}
			d X(s)= &-\{A (s)^\top X(s)+ \mathbb{E}[\bar{A} (s)^\top X(s)]-Q(s)Y^\ast(s) - \mathbb{E}[\bar Q (s)]\mathbb{E} [Y^\ast(s)]\} d s \\
			& -\{C(s)^\top X(s)+ \mathbb{E}[\bar{C}(s) ^\top X(s)]-R(s) Z^\ast(s)- \mathbb{E}[\bar{R}(s)] \mathbb{E}[Z^\ast(s)]\} d W(s), \\
			X(0)= &GY^*(0)  \in L^{2}_{\cF_0}({\RR}^n).
		\end{aligned}
		\right. \\
	\end{equation}
	
	\begin{theorem}\label{thm32}
		Let (H1) and (H2) hold,  the mean-field SDE (\ref{eq:sec3-X}) has a unique solution $X(\cdot) \in L^{2,c}_{\mathbb{F}}\big(\mathbb{R}^n\big).$
		Moreover,  there exists a constant $K >0$, which is independent of $\big(Y^\ast(\cdot),Z^\ast(\cdot)\big) $  such that
		\begin{equation}\label{statees}
			\mathbb{E}\left[\sup_{0 \leq s \leq T}|X(s)|^{2}\right] \leq K \mathbb{E}\left[ |Y^*(0)|^2+ \int^{T}_{0}|Y^\ast(s)|^2+ |Z^\ast(s)|^2ds\right].
		\end{equation}

	\end{theorem}

	\section{Main  results}\label{sec3}
	In this section,  we  collect the main results of this article. 
	
	Now we state the first main result,  which  characterizes  the optimal control for Problem (MF-BSLQ) and is referred  as  a preliminary stochastic maximum principle. 
	
	\begin{theorem}\label{preSMP}
		Let (H1) and (H2) hold. Then, Problem (MF-BSLQ) has a unique optimal control $u^*(\cdot)$. Further, $u^*(\cdot) $ is an optimal control for Problem (MF-BSLQ) if and only if the   solution $\big(X^*(\cdot),Y^*(\cdot),Z^*(\cdot)\big)\in \big( L^{2,c}_{\mathbb{F}} (\mathbb{R}^n)\big)^2\times L^2_{\mathbb{F}}(\mathbb{R}^n)$ to the following mean-field FBSDE : \ for $ s\in [0,T]$,
		\begin{equation}\label{eq:SMP0}
			\left\{
			\begin{aligned}
				d Y^*(s)= & \{A(s) Y^*(s)+\bar{A}(s) \mathbb{E}[Y^*(s)]+B(s) u^*(s)+\bar{B}(s) \mathbb{E}[u^*(s)] \\
				& +C(s) Z^*(s)+\bar{C}(s) \mathbb{E}[Z^*(s)]\} d s+Z^*(s) d W(s),\\
				d X^*(s)= &-\{A (s)^\top X^*(s)+ \mathbb{E}[\bar{A} (s)^\top X^*(s)]-  Q(s)Y^\ast(s) -\mathbb{E}[\bar Q (s)]\mathbb{E} Y^\ast(s)\} ds \\
				& -\{C(s)^\top X^*(s)+ \mathbb{E}[\bar{C}(s) ^\top X^*(s)]-R(s) Z^\ast(s)- \mathbb{E}[\bar{R}(s)] \mathbb{E}[Z^\ast(s)]\} d W(s), \\
				Y^*(T)= & \xi,\quad\quad X^*(0 )= GY^\ast(0),
			\end{aligned}
			\right. 
		\end{equation}
		satisfies the following stationary condition:
		\begin{equation}\label{cons}
			\begin{aligned}
				N(s)	u^* (s)+	  \mathbb{E} \bar N (s) \mathbb{E} u^*(s) -  B (s)^\top X^*  (s)- \mathbb{E}[ \bar B (s)^\top  X^* (s)] = 0, \;  a.e.\;  s \in [0, T], a.s.
			\end{aligned}
		\end{equation}
	\end{theorem}
	
	It is straightforward to verify that 
	\[
	\inf_{u \in \mathcal{U}} J(u(\cdot)) = \inf_{\alpha(\cdot), \beta(\cdot), \gamma(\cdot) \in \mathbb{L}^2} \inf_{u(\cdot) \in \mathcal{U}} \left\{ J(u) : \mathbb{E}[Y^{u}(\cdot)] = \alpha(\cdot), \mathbb{E}[Z^u(\cdot)] = \beta(\cdot), \text{ and } \mathbb{E}[u(\cdot)] = \gamma(\cdot) \right\}\footnote{If \(\big\{ J(u) : \mathbb{E}[Y^{u}(\cdot)] = \alpha(\cdot), \mathbb{E}[Z^u(\cdot)] = \beta(\cdot), \text{ and } \mathbb{E}[u(\cdot)] = \gamma(\cdot) \big\}=\emptyset\), then \(\inf_{u\in\mathcal{U}}J(u(\cdot))=-\infty\).} .
	\]
	Therefore,  the Problem (MF-BSLQ) will be solved into two parts: the first part  is a control problem with the constraint that the state process \(\{Y(s), Z(s)\}_{0 \leq s \leq T}\) and the control process \(\{u(s)\}_{0 \leq s \leq T}\) satisfy the conditions \(\mathbb{E}[Y(s)] = \alpha(s)\), \(\mathbb{E}[Z(s)] = \beta(s)\), and \(\mathbb{E}[u(s)] = \gamma(s)\) for all \(s \in [0, T]\), where  \(\alpha(s)\), \(\beta(s)\), and \(\gamma(s)\) are given deterministic functions,  we denote \(\bm{\eta}(\cdot) = \{\alpha(\cdot), \beta(\cdot), \gamma(\cdot)\}\), then the second is an optimization problem related to \(\bm{\eta}(\cdot) \). 
	
	Based on analysis above, we first state the step 1 .
	
	\noindent\textbf{Step 1: Solving the Constrained Control Problem}
	
	In this case, the state equation for the first part can be written as:
	for \(s \in [0, T]\),
	\begin{equation}\label{eq:BSDE1}
		\left\{\begin{aligned}
			d Y(s)= &\;\; \{A(s) Y(s)+\bar{A}(s)\alpha(s)+B(s) u(s)+\bar{B}(s) \gamma(s) \\
			& +C(s) Z(s)+\bar{C}(s) \beta(s)\} d s+Z(s) d W(s), \\
			Y(T)= &\;\; \xi;
		\end{aligned}\right.
	\end{equation}
	and the cost functional can be rewritten as
	\begin{equation}\label{eq:hatJ}
		\begin{aligned}
			\hat{J}_{\bm{\eta}}(u(\cdot)) = \mathbb{E}\Big\{& \langle G Y(0), Y(0) \rangle 
			+ \int_0^T \Big[ \langle Q(s) Y(s), Y(s) \rangle + \langle \bar{Q}(s) \alpha(s), \alpha(s) \rangle \\
			& + \langle R(s) Z(s), Z(s) \rangle + \langle \bar{R}(s) \beta(s), \beta(s) \rangle + \langle N(s) u(s), u(s) \rangle\\
			& + \langle \bar{N}(s) \gamma(s), \gamma(s) \rangle \Big] ds \Big\}.
		\end{aligned} 
	\end{equation} 
	
	Inspired by the method presented in \cite{xiong2024mean}, we will introduce three extended Lagrange multipliers, \(\lambda_1(\cdot)\), \(\lambda_2(\cdot)\), and \(\lambda_3(\cdot)\) to relax the imposed constraints. These multipliers, \(\lambda_1(\cdot)\), \(\lambda_2(\cdot)\), and \(\lambda_3(\cdot)\) are elements of \(\mathbb{L}^2\). We denote \(\bm{\lambda}(\cdot) = \{\lambda_1(\cdot), \lambda_2(\cdot), \lambda_3(\cdot)\}\).   Then we have  the following Lagrangian functional  associated with the cost functional  (\ref{eq:hatJ}):
	\begin{equation}\label{eq:j-abg1}
		\begin{aligned}
			J_{\bm{\eta}}( u(\cdot),\bm{\lambda})  
			\triangleq & \hat{J}_{\bm{\eta}}(u(\cdot))+2\langle \lambda_1,\mathbb{E}Y^u-\alpha\rangle_{\mathbb{L}^2}+2\langle \lambda_2, \mathbb{E}Z^u-\beta\rangle_{\mathbb{L}^2}
			+2\langle \lambda_3,\mathbb{E}u-\gamma\rangle_{\mathbb{L}^2}.
		\end{aligned}
	\end{equation}
	
	Therefore, step 1 can be formulated as the following Problem (a):
	
	\textbf{Problem (a):} Find a control \(u_{\bm{\eta}} (\cdot) \in \mathcal{U}\) and three ELMs \( \bm{\lambda}^{\bm{\eta} }(\cdot)=  (\lambda_1^{\bm{\eta}}(\cdot)\), \(\lambda^{\bm{\eta}}_2(\cdot)\),  \(\lambda^{\bm{\eta}}_3(\cdot) )\in (\mathbb{L}^2)^3\) such that
	
	\[
	\begin{aligned}
		&D_u J_{\bm{\eta}}( u_{\bm{\eta}}(\cdot), \bm{\lambda}^{\bm{\eta} }(\cdot)) = 0,  \quad
		D_{\lambda_1} J_{\bm{\eta}}( u_{\bm{\eta}}(\cdot), \bm{\lambda}^{\bm{\eta} }(\cdot)) = 0, \\
		&D_{\lambda_2} J_{\bm{\eta}}( u_{\bm{\eta}}(\cdot),\bm{\lambda}^{\bm{\eta} }(\cdot)) = 0, \quad
		D_{\lambda_3} J_{\bm{\eta}}( u_{\bm{\eta}}(\cdot),\bm{\lambda}^{\bm{\eta} }(\cdot)) = 0,
	\end{aligned}
	\]
	where \(D_u J_{\bm{\eta}}( u_{\bm{\eta}}(\cdot), \bm{\lambda}^{\bm{\eta} }(\cdot))\) denotes the partial derivative of \(J_{\bm{\eta}}(\cdot)\) with respect to \(u(\cdot)\), i.e., for any \(v(\cdot) \in \mathcal{U}\),
	\begin{align*}
		&\langle	D_u J_{\bm{\eta}}(u(\cdot), \bm{\lambda}(\cdot)),v(\cdot)\rangle\\
		& 	\triangleq \lim_{\epsilon\to 0^{+}}\frac{ J_{\bm{\eta}}(u(\cdot)+\epsilon v(\cdot),\bm{\lambda}(\cdot))- J_{\bm{\eta}}(u(\cdot),\bm{\lambda}(\cdot))}{\epsilon}.
	\end{align*} 
	The definitions of \(D_{\lambda_1} J_{\bm{\eta}}\), \(D_{\lambda_2} J_{\bm{\eta}}\), and \(D_{\lambda_3} J_{\bm{\eta}}\) are identical.
	
	\begin{remark}
		Based on the definition of \(J_{\bm{\eta}}(u(\cdot),\bm{\lambda}(\cdot))\) given in equation (\ref{eq:j-abg1}), and the strict convexity of the function \(\hat{J}_{\bm{\eta}}(u(\cdot),\bm{\lambda}(\cdot))\) given in Lemma \ref{lem:sigmaphi},  the strict convexity of the function \(J_{\bm{\eta}}(u(\cdot),\bm{\lambda}(\cdot))\) with respect to \(u\) holds. 
	\end{remark}
	
	Next, we solve Problem (a) in two steps. First,  fixing \(\bm{\lambda}(\cdot)  \in \big(\mathbb{L}^2\big)^3\), we determine \(\tilde{u}_{\bm{\eta}}^{\bm{\lambda}}(\cdot)\), which depends on \(\bm{\eta}(\cdot)\) and \(\bm{\lambda}(\cdot)\), such that \(D_u J_{\bm{\eta}} = 0\). Then, based on the first step, we derive  conditions  for \(\bm{\lambda}^{\bm{\eta}}(\cdot) \) to satisfy.  In fact, \(\bm{\lambda}^{\bm{\eta}}(\cdot) \)  satisfies the following equations: \(\mathbb{E}[Y^{\tilde{u}_{\bm{\eta}}^{\bm{\lambda}}}(\cdot)] = \alpha(\cdot)\), \(\mathbb{E}[Z^{\tilde{u}_{\bm{\eta}}^{\bm{\lambda}}}(\cdot)] = \beta(\cdot)\), and \(\mathbb{E}[\tilde{u}_{\bm{\eta}}^{\bm{\lambda}}(\cdot)] = \gamma(\cdot)\).

	Now, we will introduce a lemma to establish the existence and uniqueness of the optimal control for Problem (a).
	\begin{lemma}\label{lem:uabg}
		For any \(\bm{\eta} (\cdot)\in (\mathbb{L}^2)^3\) fixed, there exist a unique \(u^\ast_{\bm{\eta}}(\cdot)\in\mathcal{U}\) such that \(\mathbb{E}Y^{u^\ast_{\bm{\eta}}}(\cdot)=\alpha(\cdot)\), \(\mathbb{E}Z^{u^\ast_{\bm{\eta}}}(\cdot)=\beta(\cdot)\), \(\mathbb{E}u^\ast_{\bm{\eta}}(\cdot)=\gamma(\cdot)\), and 
		\[
		\hat{J}_{\bm{\eta}}(u^\ast_{\bm{\eta}}(\cdot))=\inf_{u(\cdot)\in\mathcal{U}}\{\hat{J}_{\bm{\eta}}(u(\cdot)): \mathbb{E}[Y^{u}(\cdot)] = \alpha(\cdot), \mathbb{E}[Z^u(\cdot)] = \beta(\cdot), \text{ and } \mathbb{E}[u(\cdot)] = \gamma(\cdot)\}.
		\]
	\end{lemma}
	
	Next, we present the following lemma to establish the existence and uniqueness of the control \( \tilde{u}_{\bm{\eta}}^{\bm{\lambda}}(\cdot) \) for the first part of Problem (a).
	\begin{lemma}\label{lem:p11}
		Let (H1) and (H2) hold. For any terminal state \(\xi \in L^2_{\mathcal{F}_T}(\mathbb{R}^n)\), there exists a unique control \(\tilde{u}_{\bm{\eta}}^{\bm{\lambda}}(\cdot) \in \mathcal{U}\) such that \(D_u J_{\bm{\eta}}( \tilde{u}_{\bm{\eta}}^{\bm{\lambda}}(\cdot), \bm{\lambda}(\cdot)) = 0\). Moreover, \(\tilde{u}_{\bm{\eta}}^{\bm{\lambda}}(\cdot)\) is optimal if and only if the   solution \(\big(\tilde{X}(\cdot), \tilde{Y}(\cdot), \tilde{Z}(\cdot)\big)\) to the following FBSDE: for \(s \in [0, T]\),
		\begin{equation}\label{eq:SMP}
			\left\{
			\begin{aligned}
				d \tilde Y(s)= &\; \{A(s) \tilde Y(s)+\bar{A}(s) \alpha(s)+B(s) \tilde u_{\bm{\eta}}^{\bm{\lambda}}(s)+\bar{B}(s) \gamma(s) \\
				& +C(s) \tilde Z(s)+\bar{C}(s) \beta(s)\} d s+\tilde Z(s) d W(s),  \\
				d \tilde X(s)= &\; -\{A(s) ^\top \tilde X(s)- Q(s)\tilde Y(s) - \lambda_1 (s)\} d s \\
				&-\{C(s)^\top \tilde X(s)- R(s) \tilde Z(s)-\lambda_2 (s) \} d W(s),  \\
				\tilde Y(T)= &\; \xi,\quad\quad \tilde X(0)= G\tilde Y(0),
			\end{aligned}
			\right. \\
		\end{equation}
		satisfies the following stationary condition:
		\begin{equation}\label{stationu2}
			\begin{aligned}
				N(s)	\tilde u_{\bm{\eta}}^{\bm{\lambda}}(s) -B(s)^\top \tilde X (s)+   \lambda_3(s)  = 0, \; a.e. s \in [0, T], a.s.
			\end{aligned}
		\end{equation}
	\end{lemma}
	
	\begin{remark}
		Under (H2), from the stationary condition (\ref{stationu2}), we have
		$$\tilde u_{\bm{\eta}}^{\bm{\lambda}}(s)=N(s)^{-1}[B(s)^\top \tilde X(s)-\lambda_3(s)].$$
		Then, we obtain
		\begin{equation}\label{eq:stationu3}
			\mathbb{E}[\tilde u_{\bm{\eta}}^{\bm{\lambda}}(s)]=\mathbb{E}\big[N(s)^{-1}[B(s)^\top \tilde X(s)- \lambda_3(s)]\big].
		\end{equation}
		Therefore, we can get \(\lambda_3^{\bm{\eta}}(\cdot)\) admits the following relation:
		\begin{equation}\label{eq:gamma1}
			\mathbb{E}\big[ N(s) ^ {-1}[B(s)^\top \tilde X (s)- \lambda_3^{\bm{\eta}}(s)]\big] = \gamma(s).
		\end{equation}
		However, we have omitted the dependence of \(\tilde{X}(\cdot)\), \(\tilde{Y}(\cdot)\), and \(\tilde{Z}(\cdot)\)  on  \(\bm{\eta} \) and \(\bm{\lambda}\) in Lemma \ref{lem:p11}. 
		It means that the expression of \(\gamma(\cdot)\) is related to \(\bm{\eta}\) and \(\bm{\lambda}\). Therefore, we cannot directly solve the ELM \(\lambda_3^{\bm{\eta}}(\cdot)\) from the relation (\ref{eq:gamma1}).
		
	\end{remark}

	In order to determine the expressions for the optimal \( \bm{\lambda}^{\bm{\eta} }(\cdot) \), we first need to obtain the expressions for the process \( (\tilde{X}(\cdot), \tilde{Y}(\cdot), \tilde{Z}(\cdot)) \). Substituting (\ref{stationu2}) into (\ref{eq:SMP}), we obtain the following system for \( s \in [0, T] \):
	\begin{equation}\label{eq1}
		\left\{\begin{aligned}
			d \tilde Y(s)= & \{A(s) \tilde Y(s)+B(s) N(s) ^ {-1}B(s)^\top \tilde X (s) -B(s) N(s) ^ {-1} \lambda_3(s)  \\
			& +C(s)  \tilde  Z(s)+\bar{A}(s)\alpha(s)+\bar{B}(s) \gamma(s) +\bar{C}(s) \beta(s)\} d s+\tilde  Z(s) d W(s),  \\
			d \tilde X(s)= &- \{A(s) ^\top \tilde X(s)- Q(s)\tilde Y(s) -  \lambda_1 (s)\} d s \\
			&-\{C(s)^\top \tilde X(s)- R(s) \tilde Z(s)- \lambda_2 (s) \} d W(s),  \\
			\tilde  	Y(T)= & \xi, \; \tilde X(0 )=  G\tilde Y(0).
		\end{aligned}\right.
	\end{equation}
	
	\begin{lemma}\label{lem:XYZ1}
		Under (H1) and (H2), for any terminal state \(\xi \in L^2_{\mathcal{F}_T}(\mathbb{R}^n)\),  the coupled system (\ref{eq1}) admits a unique   solution $\big(\tilde X(\cdot),\tilde Y(\cdot),\tilde Z(\cdot)\big)\in \big(L^{2,c}_{\FF}(\RR ^n)\big)^2\times L^2_{\FF}(\RR ^n).$
	\end{lemma}

	We have established the existence and uniqueness of \( \tilde{u}^{\bm{\lambda}}_{\bm{\eta}}(\cdot) \) for the first part of Problem (a). Based on the uniqueness of the solution of the system  (\ref{eq1}), we can express the expected value  of the processes \( (\tilde{Y}(\cdot), \tilde{Z}(\cdot)) \) by some linear operators \( \mathcal{P}_{ij} \) and  \( \mathcal{L}_{il} \) , where \(\mathcal{P}_{i0} \) maps from \(L^2_{\mathcal{F}_T} \) to \(\mathbb{L}^2\), and the others operators map from \(\mathbb{L}^2\) to \(\mathbb{L}^2\). Here, \(i=\{ 0, 1, 2\}\), \(j =\{0,1,2,3\}\) , and \(l=\{1,2,3\}\). Additionally, based on (\ref{eq:stationu3}), the expectation of \( \tilde{u}^{\bm{\lambda}}_{\bm{\eta}}(\cdot) \) can also be expressed using these operators. Thus, for all \( s \in [0, T] \),
	\begin{equation*}
		\left\{\begin{aligned}
			\mathbb{E}\tilde{Y}(s) =& \big(\mathcal{P} _{00}\xi + \mathcal{P} _{01}\alpha +  \mathcal{P} _{02}\beta+ \mathcal{P} _{03}\gamma+ \mathcal{L} _{01} \lambda_1+ \mathcal{L} _{02} \lambda_2+ \mathcal{L} _{03} \lambda_3\big)(s) \\
			\mathbb{E}\tilde{Z}(s) = &  \big(\mathcal{P} _{10}\xi + \mathcal{P} _{11}\alpha+  \mathcal{P} _{12}\beta+ \mathcal{P} _{13}\gamma+ \mathcal{L} _{11} \lambda_1 + \mathcal{L} _{12} \lambda_2+ \mathcal{L} _{13} \lambda_3\big)(s)   \\
			\mathbb{E}\tilde{u}^{\bm{\lambda}}_{\bm{\eta}}(s) =&    \big(\mathcal{P} _{20}\xi+ \mathcal{P} _{21}\alpha+  \mathcal{P} _{22}\beta+ \mathcal{P} _{23}\gamma+ \mathcal{L} _{21} \lambda_1 + \mathcal{L} _{22} \lambda _2+ \mathcal{L} _{23} \lambda_3\big)(s).
		\end{aligned}\right.
	\end{equation*}
	Therefore, the optimal \( \bm{\lambda}^{\bm{\eta} }(\cdot) \) satisfies  the following condition:
	\begin{equation}\label{eq:operators}
		\left\{\begin{aligned}
			\alpha(s)=& \big(\mathcal{P} _{00}\xi + \mathcal{P} _{01}\alpha +  \mathcal{P} _{02}\beta+ \mathcal{P} _{03}\gamma+ \mathcal{L} _{01} \lambda_1+ \mathcal{L} _{02} \lambda_2+ \mathcal{L} _{03} \lambda_3\big)(s) \\
			\beta(s)= &  \big(\mathcal{P} _{10}\xi + \mathcal{P} _{11}\alpha+  \mathcal{P} _{12}\beta+ \mathcal{P} _{13}\gamma+ \mathcal{L} _{11} \lambda_1 + \mathcal{L} _{12} \lambda_2+ \mathcal{L} _{13} \lambda_3\big)(s)   \\
			\gamma(s)=&    \big(\mathcal{P} _{20}\xi+ \mathcal{P} _{21}\alpha+  \mathcal{P} _{22}\beta+ \mathcal{P} _{23}\gamma+ \mathcal{L} _{21} \lambda_1 + \mathcal{L} _{22} \lambda _2+ \mathcal{L} _{23} \lambda_3\big)(s).
		\end{aligned}\right.
	\end{equation}
	which can be written as 
	\begin{equation}
		\begin{aligned}
			\bm{\eta}^\top=\mathcal{P} (\xi,\bm{\eta})^\top+\mathcal{L}\bm{\lambda}^\top,
		\end{aligned}
	\end{equation}
	where
	\(
	\mathcal{P}=\begin{pmatrix}
		\mathcal{P}_{00} & \mathcal{P}_{01} &\mathcal{P}_{02}&\mathcal{P}_{03}\\
		\mathcal{P}_{10}&\mathcal{P}_{11}&\mathcal{P}_{12}&\mathcal{P}_{13}\\
		\mathcal{P}_{20}&\mathcal{P}_{21}&\mathcal{P}_{22}&\mathcal{P}_{23}\\
	\end{pmatrix}
	\)
	and 
	\(
	\mathcal{L}=\begin{pmatrix}
		\mathcal{L}_{01}&\mathcal{L}_{02}&\mathcal{L}_{03}\\
		\mathcal{L}_{11}&\mathcal{L}_{12}&\mathcal{L}_{13}\\
		\mathcal{L}_{21}&\mathcal{L}_{22}&\mathcal{L}_{23}\\
	\end{pmatrix}.
	\)

	If we can uniquely solve \( \bm{\lambda}(\cdot) \) from (\ref{eq:operators}) in terms of \( \bm{\eta}(\cdot) \). Then the  problem reduces to a BSDE control problem with deterministic control variables \( \bm{\eta}(\cdot) \). While it remains unclear whether (\ref{eq:operators}) admits a unique solution, inspired by Lemma 2.6 in Xiong and Xu \cite{xiong2024mean}, we  present a lemma to demonstrate that, if  \(\bm{\lambda}(\cdot) \) is a solution to (\ref{eq:operators}), the control \( \tilde{u}^{\bm{\lambda}}_{\bm{\eta}}(\cdot) \) coincides with \( u^\ast_{ \bm{\eta} }(\cdot) \).

	\begin{lemma}\label{lem:uu}
		If \(\bm{\lambda}(\cdot)  \in (\mathbb{L}^2)^3\)  satisfies the condition (\ref{eq:operators}),   then \(\tilde{u}^{\bm{\lambda}}_{\bm{\eta}}(\cdot)=u^\ast_{\bm{\eta}}(\cdot)\).
	\end{lemma}
	
	Now  the Problem (a) has been solved, we proceed to address the second step of solving Problem (MF-BSLQ).
	
	\noindent\textbf{Step 2:} In this step, we will derive the expressions for the optimal \(\bm{\eta}^*(\cdot) = (\alpha^\ast(\cdot), \beta^\ast(\cdot), \gamma^\ast(\cdot) )\). Based on these considerations, the state equation considered in this step is as follows: for all \( s \in [0, T] \),
	\begin{equation}\label{eq:stateabg}
		\left\{\begin{aligned}
			d Y(s)= &\;\; \{A(s) Y(s)+\bar{A}(s)\alpha(s)+B(s) u^\ast_{\bm{\eta}}(s)+\bar{B}(s) \gamma(s) \\
			& +C(s) Z(s)+\bar{C}(s) \beta(s)\} d s+Z(s) d W(s), \\
			Y(T)= &\;\; \xi;
		\end{aligned}\right.
	\end{equation}
	and the cost functional related to \(\bm{\eta}(\cdot)\) is given by
	\begin{equation}\label{eq:costabg}
		\begin{aligned}
			\hat{J}_{\bm{\eta}}(u^\ast_{\bm{\eta}}(\cdot)) = \mathbb{E}\Big\{& \langle G Y(0), Y(0) \rangle 
			+ \int_0^T \Big[ \langle Q(s) Y(s), Y(s) \rangle + \langle \bar{Q}(s) \alpha(s), \alpha(s) \rangle \\
			& + \langle R(s) Z(s), Z(s) \rangle + \langle \bar{R}(s) \beta(s), \beta(s) \rangle + \langle N(s) u^\ast_{\bm{\eta}}(s), u^\ast_{\bm{\eta}}(s) \rangle\\
			& + \langle \bar{N}(s) \gamma(s), \gamma(s) \rangle \Big] ds \Big\}.
		\end{aligned}
	\end{equation} 
	
	And the problem we consider in this step is as follows:
	
	\noindent\textbf{Problem (b):} \text{ Find } \(\bm{\eta}^*(\cdot)\in (\mathbb{L}^2)^3\) such that
	\[
	\hat{J}_{\bm{\eta}^*}(u^\ast_{\bm{\eta}^*}(\cdot)) = \inf_{\bm{\eta}(\cdot)\in(\mathbb{L}^2)^3} \hat{J}_{\bm{\eta}}(u^\ast_{\bm{\eta}}(\cdot)).
	\]

	Based on Lemmas \ref{lem:uabg}, \ref{lem:XYZ1}, and \ref{lem:uu}, we can express \(u^\ast_{\bm{\eta}}(\cdot)\) as a linear combination of the operators \(\cO_{i2}\) for \(i \in \{1, 2, 3, 4\}\), acting on \(\xi\), \(\alpha\), \(\beta\), and \(\gamma\), respectively, as follows
	\begin{equation}\label{eq:uastoperator}
		u^\ast_{\bm{\eta}}(\cdot)=(\cO_{21}\xi)(\cdot)+(\cO_{22} \alpha)(\cdot)+(\cO_{23} \beta)(\cdot)+(\cO_{24}\gamma)(\cdot).
	\end{equation}
	
	Therefore, based on the state equation (\ref{eq:stateabg}), the state process can also be expressed as a linear combination, namely:
	\begin{equation}\label{eq:YZoperator}
		\left\{
		\begin{aligned}
			Y(\cdot)&=(\mathcal{O}_{01}\xi)(\cdot)+(\mathcal{O}_{02}\alpha)(\cdot)+(\mathcal{O}_{03} \beta)(\cdot)+(\mathcal{O}_{04} \gamma)(\cdot),\\
			Z(\cdot)&=(\mathcal{O}_{11}\xi)(\cdot)+(\mathcal{O}_{12}\alpha)(\cdot)+(\mathcal{O}_{13} \beta)(\cdot)+(\mathcal{O}_{14} \gamma)(\cdot),
		\end{aligned}
		\right.
	\end{equation}
	and 
	\begin{equation}\label{eq:YZoperator1}
		\begin{aligned}
			Y(0)&=\mathcal{O}_{31}\xi+\mathcal{O}_{32}\alpha+\mathcal{O}_{33} \beta+\mathcal{O}_{34}\gamma,\\
		\end{aligned}
	\end{equation}

	Based on the previous results, we now state the following lemma, which provides the necessary and sufficient conditions for the optimal \( \bm{\eta}^*(\cdot)  \) of Problem (b).
	\begin{lemma}\label{lem:conditions}
		Suppose that Assumptions (H1) and (H2) hold. Then, the necessary and sufficient conditions for  \( \bm{\eta}^*(\cdot)  \)  to be optimal for Problem (b) are
	\begin{align}\label{eq:conditions}
		\mathbf{I}_0\big(\mathcal{O}^\top\mathcal{T}\mathcal{O}+\mathcal{W}\big)(\xi,\bm{\eta}^\ast)^\top=(0,0,0,0)^\top,
	\end{align}
	where \(\mathbf{I}_0= 
	\begin{pmatrix}
		0&0&0&0\\
		0&1&0&0\\
		0&0&1&0\\
		0&0&0&1
	\end{pmatrix}\), \(\mathcal{O}=
	\begin{pmatrix}
		\mathcal{O}_{01}&\mathcal{O}_{02}&\mathcal{O}_{03}&\mathcal{O}_{04}\\
		\mathcal{O}_{11}&\mathcal{O}_{12}&\mathcal{O}_{13}&\mathcal{O}_{14}\\
		\mathcal{O}_{21}&\mathcal{O}_{22}&\mathcal{O}_{23}&\mathcal{O}_{24}\\
		\mathcal{O}_{31}&\mathcal{O}_{32}&\mathcal{O}_{33}&\mathcal{O}_{34}
	\end{pmatrix}
	\),
	\(\mathcal{T}=
	\begin{pmatrix}
		Q&0&0&0\\
		0&R&0&0\\
		0&0&N&0\\
		0&0&0&G
	\end{pmatrix}\), and \(\mathcal{W}=
	\begin{pmatrix}
		0&0&0&0\\
		0&\bar{Q}&0&0\\
		0&0&\bar{R}&0\\
		0&0&0&\bar{N}
	\end{pmatrix}\).
	
\end{lemma}

Now we are ready to present the main theorem of this article.
\begin{theorem}\label{main}
	Let Assumptions (H1) and (H2) hold. Then, the unique optimal control \( u^\ast(\cdot) \) for Problem (MF-BSLQ) is given by (\ref{stationu2}) with $\tilde{X}(\cdot)$ replaced by $X^*(\cdot)$.  Moreover, \( (X^\ast(\cdot), Y^\ast(\cdot), Z^\ast(\cdot)) \) is the unique solution of the FBSDE (\ref{eq:SMP}) with \( \alpha(\cdot) \), \( \beta(\cdot) \), and \( \gamma(\cdot) \) replaced by \( \alpha^\ast(\cdot) \), \( \beta^\ast(\cdot) \), and \( \gamma^\ast(\cdot) \), respectively. The Lagrange multipliers \( \lambda_1(\cdot) \), \( \lambda_2(\cdot) \), and \( \lambda_3(\cdot) \) are obtained by solving (\ref{eq:operators}) with \( \alpha(\cdot) \), \( \beta(\cdot) \), and \( \gamma(\cdot) \) replaced by their optimal counterparts \( \alpha^\ast(\cdot) \), \( \beta^\ast(\cdot) \), and \( \gamma^\ast(\cdot) \). Furthermore, the optimal deterministic processes \( \alpha^\ast(\cdot) \), \( \beta^\ast(\cdot) \), and \( \gamma^\ast(\cdot) \) are obtained by solving (\ref{eq:conditions}).
\end{theorem}


\section{ The proof to Theorem \ref{preSMP}}\label{sec4}
In this section,  we give a lemma to state the existence and uniqueness of the optimal control for the Problem (MF-BSDE).  Finally,  we prove Theorem \ref{preSMP} at the end of this section.

Before presenting the proof of Theorem \ref{preSMP},  we now give the following lemma to ensure the  uniqueness of the optimal control 
for the Problem (MF-BSLQ). The existence  is ensured by Mazur's theorem, see Theorem 5.2 of \cite{yong2012stochastic},  we only prove the uniqueness.
\begin{lemma}\label{lem:strictc}
	Let (H1) and (H2) hold,  the cost functional $ J\big( u(\cdot)\big) $ is strictly convex. 
\end{lemma}
\begin{proof}
	For any \(u\in\cU[0,T]\), consider the following BSDE\(:\) \ for \(s\in[0,T]\),
	\begin{equation}\label{stateY0u}
		\left\{\begin{aligned}
			d Y^{0,u}(s)= & \{A(s) Y^{0,u}(s)+\bar{A}(s) \mathbb{E}[Y^{0,u}(s)]+B(s) u(s)+\bar{B}(s) \mathbb{E}[u(s)] \\
			& +C(s) Z^{0,u}(s)+\bar{C}(s) \mathbb{E}[Z^{0,u}(s)]\} d s+Z^{0,u}(s) d W(s), \\
			Y^{0,u}(T)= & 0.
		\end{aligned}\right.
	\end{equation}
	By Theorem \ref{preSMP}, the above BSDE admits a unique solution 
	$$(Y^{0,u},Z^{0,u})\in L^2_\FF \big(\Omega;C(0,T;\RR^n)\big)\times L^2_\FF({\RR}^n).$$ 
	By the linearity of BSDE (\ref{stateY0u}) and Theorem \ref{thm31},  we can define two bounded linear operators $$\cH_1:\cU[0,T]\rightarrow L^2_\FF \big(\Omega;C(0,T;\RR ^n)\big)\times  L^2_\FF({\RR}^n),$$ and \(\cH_2:\cU[0,T]\rightarrow L^2_{\cF_0}(\RR ^n)\) as follows:
	\begin{align*}
		\cH_1 u :=(Y^{0,u},Z^{0,u})^\top,\quad\quad \cH_2 u : =Y^{0,u}(0).
	\end{align*}
	Also we can define the bounded linear operators $$\cH_3:L^2_{\cF_T}(\RR ^n)\rightarrow L^2_\FF \big(\Omega;C(t,T;\RR^n)\big)\times L^2_\FF({\RR}^n),$$ and  \(\cH_4:L^2_{\cF_T}(\RR ^n)\rightarrow L^2_{\cF_0}\big(\RR ^n\big)\) as follows:
	\begin{align*}
		\cH_3 \xi :=(Y^{\xi,0},Z^{\xi,0})^\top,\quad\quad \cH_4\xi :=Y^{\xi,0}(0),  
	\end{align*}
	where \((Y^{\xi,0},Z^{\xi,0})^\top\) is the   solution of the following BSDE:\ for \(s\in[0,T]\),
	\begin{equation}\label{stateYxi0}
		\left\{\begin{aligned}
			d Y^{\xi,0}(s)= & \{A(s) Y^{\xi,0}(s)+\bar{A}(s) \mathbb{E}[Y^{\xi,0}(s)] +C(s) Z^{\xi,0}(s)+\bar{C}(s) \mathbb{E}[Z^{\xi,0}(s)]\} d s\\
			&+Z^{\xi,0}(s) d W(s),  \\
			Y^{\xi,0}(T)= & \xi.
		\end{aligned}\right.
	\end{equation}
	It's obvious that for any \((\xi,u)\in L^2_{\cF_T}(\RR ^n)\times \cU[0,T]\), the sum \((Y^{0,u}+Y^{\xi,0},Z^{0,u}+Z^{\xi,0})^\top\) is the solution of (\ref{state1}). And by the uniqueness of the   solution of (\ref{state1}), we obtain
	\begin{align}
		(Y,Z)^\top=(Y^{0,u}+Y^{\xi,0},Z^{0,u}+Z^{\xi,0})^\top=\cH_1 u+\cH_3 \xi. 
	\end{align}
	Moreover, \(Y(0)=\cH_2 u+\cH_4 \xi\).
	
	Since that \(\mathbb{E}[(Y,Z)^\top]=\mathbb{E}[\cH_1 u+\cH_3\xi]\), 
	we can also define two another bounded linear operators
	$$\cN_1:\cU[0,T]\rightarrow \LL^2 \times \LL^2$$  and 
	$$\cN_2:L^2_{\cF_T}(R ^n)\rightarrow \LL^2 \times \LL^2 $$ such that
	\begin{align*}
		\cN_1 u :=\mathbb{E}[\cH_1 u],\quad\quad \cN_2 \xi :=\mathbb{E}[\cH_3 \xi].
	\end{align*}
	Therefore,\  \(\mathbb{E}[(Y,Z)^\top ]=\cN_1u+\cN_2 \xi\).
	
	Additional, define linear  bounded operator \(\cN_3 :\cU [0,T]\rightarrow \LL^2\) 
	\begin{align*}
		\cN_3 u :=\mathbb{E}[u],
	\end{align*}
	then define
	\begin{align*}
		\mathcal{Q}_1 :=
		\begin{pmatrix}
			Q&0\\
			0& R
		\end{pmatrix},\quad\quad
		\mathcal{Q}_2 :=
		\begin{pmatrix}
			\bar Q&0\\
			0&\bar R
		\end{pmatrix}.
	\end{align*}
	
	Therefore, the cost functional (\ref{cost}) can be rewritten as follows:
	\begin{align*}
		J( u(\cdot)) = \mathbb{E}\bigg\{ & \langle G Y(0), Y(0)\rangle 
		+\int_0^T\Big[\Big\langle 
		\begin{pmatrix}
			Q&0\\
			0& R
		\end{pmatrix}\begin{pmatrix}
			Y\\
			Z
		\end{pmatrix} 	,
		\begin{pmatrix}
			Y\\
			Z
		\end{pmatrix} \Big\rangle \\
		&+\Big\langle 
		\begin{pmatrix}
			\bar Q&0\\
			0&\bar R
		\end{pmatrix}	\begin{pmatrix}
			\mathbb{E}[Y]\\
			\mathbb{E}[Z]
		\end{pmatrix},
		\begin{pmatrix}
			\mathbb{E}[Y]\\
			\mathbb{E}[Z]
		\end{pmatrix}\Big\rangle\\
		&+\langle N u, u \rangle   + \langle \bar N  \mathbb{E} [u], \mathbb{E} [u]\rangle  \Big] ds \bigg\}\\
		=&\langle G (\cH_2 u+\cH_4 \xi),\cH_2 u+\cH_4 \xi\rangle_{\LL^2_{\FF}}+\langle  \mathcal{Q}_1(\cH_1 u+\cH_3\xi),(\cH_1 u+\cH_3\xi)\rangle_{\LL^2_{\FF}} \\
		&+\langle\mathcal{Q}_2 (\cN_1 u+\cN_2\xi),(\cN_1 u+\cN_2\xi)\rangle _{\LL^2_{\FF}}+\langle Nu,u\rangle _{\LL^2_{\FF}}+\langle \bar N \cN_3 u,\cN_3 u\rangle_{\LL^2_{\FF}}\\
		=&\langle (\cH_2^* G\cH_2 +\cH_1^* \mathcal{Q}_1 \cH_1 +\cN_1^*\mathcal{Q}_2 \cN_1+N+\cN_3^*\bar N\cN_3)u,u\rangle_{\LL^2_{\FF}}\\
		&+\langle (\cH_4^*G\cH_4+\cH_3^*\mathcal{Q}_1\cH_3 +\cN_2^*\mathcal{Q}_2\cN_2 )\xi,\xi\rangle _{\LL^2_{\FF}} \\
		&+2\langle (\cH_4^*G\cH_2+\cH_3^*\mathcal{Q}_1\cH_1+\cN_2^*\mathcal{Q}_2\cN_1)u,\xi\rangle_{\LL^2_{\FF}},
	\end{align*}
	where $\langle \cdot, \cdot \rangle _{\LL^2_{\FF}}$ denotes the inner  product in the space $\LL^2_{\FF}(\RR ^n)$.
	Based on the assumption (H2), we have
	\begin{align*}
		\langle (\cH_2^* G\cH_2 +\cH_1^* \mathcal{Q}_1 \cH_1 +\cN_1^*\mathcal{Q}_2 \cN_1+N+\cN_3^*\bar N\cN_3)u,u\rangle_{\LL^2_{\FF}}  \geq \delta \mathbb{E}\int_0^T |u(s)|^2ds.
	\end{align*}
	Therefore, the mapping \(u\rightarrow J(u)\) is strictly convex.
\end{proof}

Let $\big(u^\ast(\cdot), Y^\ast(\cdot), Z^\ast(\cdot)\big)$ be the optimal state control pair which  satisfies the state equation  (\ref{state1}), and $\big(Y^\epsilon(\cdot), Z^\epsilon(\cdot)\big)$ be the  solution corresponding to the control $u^{\epsilon}(\cdot)=u^\ast(\cdot)+\epsilon v (\cdot)$, where  $\epsilon \in \RR$,  $v (\cdot) \in \mathcal{U}[0, T] $.  Then, we introduce the variation equation:\ for $s \in[0, T]$,
\begin{equation}\label{eqY1}
	\left\{\begin{aligned}
		d Y_1(s)= &\{A(s) Y_1(s)+\bar{A}(s) \mathbb{E}[Y_1(s)]+B(s) v(s)+\bar{B}(s) \mathbb{E}[v(s)] \\
		& +C(s) Z_1(s)+\bar{C}(s) \mathbb{E}[Z_1(s)]\} d s+Z_1(s) d W(s),  \\
		Y_1(T)= & 0.
	\end{aligned}\right.
\end{equation}

It's easy to verify that 
\begin{align*}
	\big(Y^\epsilon(s)-Y^*(s)\big)= 	\epsilon Y_1(s),\text{ and } \big(Z^\epsilon(s)-Z^*(s)\big)=\epsilon  Z_1(s),\text{a.s.},s\in[0,T],\text{a.e.}
\end{align*}
Then, by Theorem \ref{thm31}, we have 
\begin{align*}
	&\mathbb{E}\Big[\sup_{0\leq s\leq T}\{|Y^\epsilon(s)-Y^*(s)|^2\}+\int_0^T|Z^\epsilon(s)-Z^*(s)|^2ds\Big]\\
	&=\epsilon^2 \mathbb{E}[\sup_{0\leq s\leq T}|Y_1(s)|^2+\int_0^T |Z_1(s)|^2ds]\\
	&\leq K\epsilon^2  \mathbb{E}[\int_0^T |v(s)|^2ds]\\
	&\leq K\epsilon^2, 
\end{align*}
where \(K > 0\) is a constant which can be different from line to line.

Now, we are in the position to present the proof of Theorem \ref{preSMP}.

{\em	Proof of Theorem \ref{preSMP}}: 
As we mentioned earlier, the existence of an optimal control
follows from the same argument as in the proof of Theorem 5.2 of \cite{yong2012stochastic} using Mazur's  theorem. To prove the uniqueness, suppose that $u_1(\cdot)$ and $u_2(\cdot)$ are two different optimal controls. Let $u_3(\cdot)=\frac{1}{2}(u_1(\cdot)+u_2(\cdot))$. By strict convexity of the cost functional proved in the last lemma, we have
\[J(u_1(\cdot))=J(u_2(\cdot))\le J(u_3(\cdot))<\frac12(J(u_1(\cdot))+J(u_2(\cdot)))=J(u_1(\cdot)).\]
This contradiction implies the uniqueness of the optimal control. 

Now we prove the necessity. 
As $u^*(\cdot)$ is an optimal control for Problem (MF-BSLQ), $\big(Y^\ast(\cdot), Z^\ast(\cdot)\big)$ be the corresponding optimal state, we have
$$
\begin{aligned}
	0  \leq &  \lim _{\epsilon \rightarrow 0} \frac{J\left( u^\epsilon(\cdot)\right)-J\left(u^*(\cdot)\right)}{\epsilon} \\
	=& 2 \mathbb{E}\Big\{\int_0^T(\langle Q Y^*, Y_1\rangle+\langle \bar Q \mathbb{E} Y^*, \mathbb{E} Y_1\rangle+\langle N u^*, v\rangle+ \langle \bar N  \mathbb{E} u^*,  \mathbb{E} v\rangle \\
	&+ \langle R Z^\ast, Z_1\rangle+\langle\bar{R} \mathbb{E}Z^\ast, \mathbb{E}Z_1\rangle 
	) d s+\langle G Y^\ast(0), Y_1(0)\rangle\Big\} \\
	=& 2 \mathbb{E}\Big\{\int_0^T\Big(\langle Q Y^* + \mathbb{E}  \bar Q \mathbb{E} Y^*, Y_1\rangle+ \langle R Z^\ast+ \mathbb{E}\bar{R} \mathbb{E}Z^\ast, Z_1\rangle\\
	&+  \langle  N  u^* + \mathbb{E} \bar N  \mathbb{E} u^*  , v\rangle 
	\Big) d s+\langle G Y^\ast(0), Y_1(0)\rangle\Big\}.
\end{aligned}
$$
Applying It\^{o}'s formula to $\langle X^\ast, Y_1\rangle$, we have
$$
\begin{aligned}
	d \langle X^\ast, Y_1\rangle = & -\langle  A ^\top X^\ast+ \mathbb{E}[\bar{A}^\top X^\ast]- QY^\ast- \mathbb{E}\bar Q \mathbb{E} Y^\ast, Y_1 \rangle   d s \\
	& - \langle C^\top X^\ast+ \mathbb{E}[\bar{C} ^\top X^\ast]-R Z^\ast-\mathbb{E}\bar{R} \mathbb{E}Z^\ast , Y_1 \rangle d W(s)\\
	& +\langle X^\ast, AY_1+\bar{A} \mathbb{E}Y_1+B v+\bar{B} \mathbb{E}v+CZ_1+\bar{C} \mathbb{E}[Z_1] \rangle ds \\
	&+ \langle X^\ast, Z_1 \rangle dW(s) - \langle C^\top X^\ast+ \mathbb{E}[\bar{C} ^\top X^\ast]-R Z^\ast- \mathbb{E}\bar{R} \mathbb{E}Z^\ast, Z_1 \rangle  ds  \\
	= & \langle -\mathbb{E}[\bar{A}^\top X^\ast]+ QY^\ast+ \mathbb{E}\bar Q \mathbb{E} Y^\ast, Y_1\rangle   ds +\langle X^\ast, \bar{A} \mathbb{E}Y_1+B v+\bar{B} \mathbb{E}v+\bar{C} \mathbb{E}Z_1 \rangle ds \\
	&+  \langle - \mathbb{E}[\bar{C} ^\top X^\ast]+R Z^\ast+ \mathbb{E}\bar{R} \mathbb{E}Z^\ast, Z_1 \rangle  ds  + (...)dW(s).
\end{aligned}
$$
Therefore, 
$$
\begin{aligned}
	\mathbb{E}\langle G Y^\ast(0), Y_1(0)\rangle =&  - \mathbb{E} \int ^T_0 \Big(\langle -\mathbb{E}[\bar{A}^\top X^\ast]+ QY^\ast+ \mathbb{E}\bar Q \mathbb{E} Y^\ast, Y_1\rangle  \\
	& +\langle X^\ast, \bar{A} \mathbb{E}Y_1+B v+\bar{B} \mathbb{E}v+\bar{C} \mathbb{E}Z_1 \rangle \\
	&+ \langle -\mathbb{E}[\bar{C} ^\top X^\ast]+R Z^\ast+ \mathbb{E}\bar{R} \mathbb{E}Z^\ast, Z_1 \rangle\Big) ds,
\end{aligned}
$$
namely, 
$$
\begin{aligned}
	0 & \leq \lim _{\epsilon \rightarrow 0} \frac{J\left( u^\epsilon(\cdot)\right)-J\left( u^*(\cdot)\right)}{\epsilon} \\
	& \leq 2 \mathbb{E}\Big\{\int_0^T  \langle  N  u^* + \mathbb{E} \bar N  \mathbb{E} u^* - B^\top X^\ast- \mathbb{E}[\bar B ^\top X^\ast]  , v\rangle  ds. 
\end{aligned}
$$
Since $v$ is arbitrary,\   it follows that
$$
\begin{aligned}
	N (s) u^* (s)+ \mathbb{E} \bar N (s) \mathbb{E} u^*(s) - B (s)^\top X^\ast(s)- \mathbb{E}[ \bar B (s)^\top  X^\ast(s)]  = 0,\;  a.e. \; s \in [0, T], \; a.s.
\end{aligned}
$$

For the sufficiency, we will now verify that if \(u^*(\cdot)\) satisfies the above equation, then \(u^*(\cdot)\) is an optimal control, i.e. 
\begin{align*}
	J(u^*(\cdot))\leq J(u^*(\cdot)+\epsilon v(\cdot)), \quad \forall \epsilon\in \mathbb{R},\quad v\in \mathcal{U}[0,T].
\end{align*}

Let \( (Y_1(\cdot), Z_1(\cdot)) \) be the solution of (\ref{eqY1}) corresponding to the control \( v(\cdot) \). We obtain 
\begin{align*}
	&J(u^*(\cdot)+\epsilon v(\cdot))-J(u^*(\cdot))\\
	=&\epsilon^2 \mathbb{E}\Big\{\langle GY_1(0),Y_1(0)\rangle +\int_0^T \langle QY_1,Y_1\rangle + \langle \bar{Q}\mathbb{E} Y_1,\mathbb{E}Y_1\rangle\\
	& +\langle R Z_1,Z_1\rangle +\langle \bar{R}\mathbb{E}Z_1,\mathbb{E}Z_1\rangle +\langle N v,v \rangle +\langle \bar{N}\mathbb{E}[v],\mathbb{E}[v]\rangle\Big\}\\
	&+2 \epsilon \mathbb{E}\Big\{\int_0^T\Big(\langle Q Y^* + \mathbb{E}  \bar Q \mathbb{E} Y^*, Y_1\rangle+ \langle R Z^\ast+ \mathbb{E}\bar{R} \mathbb{E}Z^\ast, Z_1\rangle\\
	&+  \langle  N  u^* + \mathbb{E} \bar N  \mathbb{E} u^*  , v\rangle 
	\Big) d s+\langle G Y^\ast(0), Y_1(0)\rangle\Big\}\\
	&=\epsilon^2  \mathbb{E}\Big\{\langle GY_1(0),Y_1(0)\rangle +\int_0^T \langle QY_1,Y_1\rangle + \langle \bar{Q}\mathbb{E} Y_1,\mathbb{E}Y_1\rangle\\
	& +\langle R Z_1,Z_1\rangle +\langle \bar{R}\mathbb{E}Z_1,\mathbb{E}Z_1\rangle +\langle N v,v \rangle +\langle \bar{N}\mathbb{E}[v],\mathbb{E}[v]\rangle\Big\}.
\end{align*}

Based on the assumption \((H2)\), we obtain that 
\begin{align*}
	&\mathbb{E}\Big\{\langle GY_1(0),Y_1(0)\rangle +\int_0^T \langle QY_1,Y_1\rangle + \langle \bar{Q}\mathbb{E} Y_1,\mathbb{E}Y_1\rangle\\
	& +\langle R Z_1,Z_1\rangle +\langle \bar{R}\mathbb{E}Z_1,\mathbb{E}Z_1\rangle +\langle N v,v \rangle +\langle \bar{N}\mathbb{E}[v],\mathbb{E}[v]\rangle\Big\}\geq 0,
\end{align*}
which implies
\begin{align*}
	J(u^*(\cdot)+\epsilon v(\cdot))-J(u^*(\cdot))\geq 0,\quad\quad \forall \epsilon\in \mathbb{R},\quad v\in \mathcal{U}[0,T].
\end{align*}

In summary,  the proof is complete.
\qed
\begin{remark}
	Under  (H2),  we can  obtain a representation of the optimal control $u^\ast (\cdot) $ for the Problem (MF-BSDE).  More precisely,\  from the  stationary condition (\ref{cons}),  we have
	$$
	\begin{aligned}
		\mathbb{E}  	u^* (s)+	\mathbb{E}  N (s) ^{-1} \mathbb{E} \bar N (s) \mathbb{E} u^*(s) - \mathbb{E}  [N (s) ^{-1} B (s)^\top X^\ast(s)]-\mathbb{E}  N (s) ^{-1}  \mathbb{E}[ \bar B (s)^\top  X^\ast(s)]  = 0.
	\end{aligned}
	$$
	Since $I_n+ \mathbb{E}  N (s) ^{-1} \mathbb{E} \bar N (s)$ is an invertible matrix, then we have
	$$
	\begin{aligned}
		\mathbb{E}  	u^* (s) =  \big(I_n+ \mathbb{E}  N (s) ^{-1} \mathbb{E} \bar N (s) \big)^{-1}\Big\{\mathbb{E}  [N (s) ^{-1} B(s)^\top X^\ast(s)]+\mathbb{E}  N (s) ^{-1}  \mathbb{E}[ \bar B (s) ^\top X^\ast(s)] \Big\}.
	\end{aligned}
	$$
	Therefore, 
	\begin{equation}\label{optimalc}
		\begin{aligned}
			u^* (s) =& -N (s) ^{-1} 	\mathbb{E}\bar N (s)   \big(I_n+ \mathbb{E}  N (s) ^{-1} \mathbb{E} \bar N (s) \big)^{-1}\Big\{\mathbb{E}  [N (s) ^{-1} B (s)^\top X^\ast(s)]\\
			&+\mathbb{E}  N (s) ^{-1}  \mathbb{E}[ \bar B (s)^\top X^\ast(s)] \Big\} +N (s) ^{-1}  \Big\{B (s)^\top X^\ast(s)+ \mathbb{E}[ \bar B (s) ^\top X^\ast(s)] \Big\}.
		\end{aligned}
	\end{equation}
	From  above, it can be seen that if we substitute the optimal control (\ref{optimalc}) into the system (\ref{eq:SMP0}), we will obtain a highly complex optimality system, making it impossible to decouple and obtain an explicitly  representation of  optimal control.  However, Xiong and Xu \cite{xiong2024mean} gave us a procedure to find the optimal control, in which they provided the  extended LaGrange multiplier method to study the stochastic mean-field problem with random coefficients.  Inspired by \cite{xiong2024mean}, we will first  study the  constrained  BSLQ problem with random coefficients using the extended LaGrange multiplier  in the next section. 
\end{remark}

\section{Solving Problem (a)}\label{sec5}
In this section, we focus on solving Problem (a).  We will first establish the proofs of Lemmas \ref{lem:uabg} - \ref{lem:uu}. Then we will decouple the system (\ref{eq:SMP}) at the end of this section. 

The proof of Lemma \ref{lem:uabg} is
similar to the proof of Lemma 2.2 in \cite{xiong2024mean}, and  the proofs of Lemma  \ref{lem:XYZ1} and Lemma \ref{lem:uu} are analogous to those of Lemma 2.4 and Lemma 2.6 in \cite{xiong2024mean}, we omit the detailed derivations here for brevity.  

Now we  give the proof of Lemma \ref{lem:p11}.

\noindent\textit{Proof of Lemma \ref{lem:p11}:}
The existence and uniqueness of \(\tilde{u}_{\bm{\eta}}^{\bm{\lambda}}(\cdot) \) can be established by employing the approach outlined in Lemma \ref{lem:strictc}, which shows that \(J_{\bm{\eta}}(u(\cdot),\bm{\lambda})\) is strictly convex with respect to \(u(\cdot)\).  Now we prove the  remainder. 

Let \((\tilde{Y}(\cdot), \tilde{Z}(\cdot))\) denote the solution of (\ref{eq:BSDE1}) with respect to \(\tilde{u}_{\bm{\eta}}^{\bm{\lambda}}(\cdot) \), and let \((Y^\epsilon(\cdot), Z^\epsilon(\cdot))\) represent the state corresponding to the control \(u^\epsilon(\cdot) = \tilde{u}_{\bm{\eta}}^{\bm{\lambda}}(\cdot) + \epsilon v(\cdot)\), where \(\epsilon >0\) and \(v(\cdot) \in \mathcal{U}[0,T]\). 

The associated  variation equation  is given as follows for \(s \in [0, T]\):  
\begin{equation}\label{eq:Y1}
	\left\{\begin{aligned}
		d Y_1(s)= & \{A(s) Y_1(s)+B(s) v(s) 
		+C(s) Z_1(s)\} d s+Z_1(s) d W(s), \\
		Y_1(T)= & 0.
	\end{aligned}\right.
\end{equation}

So we have 
\begin{equation}
	\begin{aligned}
		0 \leq & \lim_{\epsilon \to 0} \frac{J_{\bm{\eta}}( u^\epsilon(\cdot),\bm{\lambda}) -J_{\bm{\eta}}( \tilde{u}_{\bm{\eta}}^{\bm{\lambda}}(\cdot),\bm{\lambda})  }{\epsilon}\\
		=&  2  \mathbb{E}\Big\{  \langle G \tilde Y(0), Y_1(0)\rangle + \int_0^T[\langle Q(s) \tilde Y(s), Y_1(s)\rangle   +\langle R(s) \tilde Z(s), Z_1(s)\rangle\\
		&+ \langle N(s)\tilde{u}_{\bm{\eta}}^{\bm{\lambda}}(s), v(s) \rangle   ] ds + 	\int_0^T \langle   \lambda_1 (s), Y_1(s) \rangle ds 		+\int_0^T \langle  \lambda_2 (s), Z_1(s) \rangle ds\\
		&+\int_0^T \langle  \lambda_3(s), v(s) \rangle ds  \Big\},\\
		=&  2  \mathbb{E}\Big\{  \langle G \tilde Y(0), Y_1(0)\rangle +  \int_0^T[\langle Q(s) \tilde Y(s)+  \lambda_1 (s), Y_1(s)\rangle   +\langle R(s) \tilde Z(s)+ \lambda_2 (s), Z_1(s)\rangle \\
		&+  \langle N(s)\tilde{u}_{\bm{\eta}}^{\bm{\lambda}}(s)+  \lambda_3(s), v(s) \rangle   ] ds. 
	\end{aligned}
\end{equation}

Now we apply It\^{o}'s formula to $ \langle	\tilde X(\cdot), Y_1(\cdot) \rangle$,  we obtain 
$$
\begin{aligned}
	d \langle \tilde X(s), Y_1(s) \rangle = &- \langle A(s) ^\top \tilde X(s)-Q(s)\tilde Y(s) -  \lambda_1 (s), Y_1(s) \rangle d s \\
	&-\langle  C(s)^\top \tilde X(s)- R(s) \tilde Z(s)- \lambda_2 (s ) , Z_1(s)\rangle d s\\
	&+ \langle \tilde X(s), A(s) Y_1(s)+B(s) v(s) 
	+C(s) Z_1(s) \rangle ds +(...)dW(s).
\end{aligned}
$$

Therefore, 
$$
\begin{aligned}
	\mathbb{E} \langle \tilde X(0), Y_1(0) \rangle = &\int^T_0  \Big\{ -\langle Q(s)\tilde Y(s) + \lambda_1 (s), Y_1(s) \rangle  
	-\langle  R(s) \tilde Z(s)+ \lambda_2 (s ) , Z_1(s)\rangle \\
	&-\langle \tilde X(s), B(s) v(s) 
	\rangle \Big\} ds.
\end{aligned}
$$
Namely,  we have
\begin{equation}
	\begin{aligned}
		0 \leq & \lim_{\epsilon \to 0} \frac{J_{\bm{\eta}}( u^\epsilon(\cdot),\bm{\lambda}) -J_{\bm{\eta}}( \tilde{u}_{\bm{\eta}}^{\bm{\lambda}}(\cdot),\bm{\lambda})  }{\epsilon}\\
		\leq & 2  \mathbb{E}  \int_0^T    \langle N(s) \tilde{u}_{\bm{\eta}}^{\bm{\lambda}}(s)+  \lambda_3(s)-B(s)^\top \tilde  X(s), v(s) \rangle   ds. 
	\end{aligned}
\end{equation}
Therefore, by the arbitrariness of \(v(\cdot)\), we have 
$$
N(s) \tilde{u}_{\bm{\eta}}^{\bm{\lambda}}(s)-B(s)^\top \tilde X (s)+ \lambda_3(s)=0,\text{ a.e. }s\in[0,T],\text{ a.s.}
$$

We have completed the proof of necessity. We will now proceed to prove sufficiency, which requires verifying that if \(\tilde{u}_{\bm{\eta}}^{\bm{\lambda}}(\cdot)\) satisfies the above equation, it is indeed an optimal control. 

For any \(\epsilon > 0\) and \(v(\cdot) \in \mathcal{U}[0,T]\), let \((Y_1(\cdot), Z_1(\cdot))\) denote the unique solution to (\ref{eq:Y1}) corresponding to the control \(v(\cdot)\). Furthermore, let the state associated with the control \(u^\epsilon(\cdot)\) be given by \((Y^\epsilon(\cdot), Z^\epsilon(\cdot)) = (\tilde{Y}(\cdot) + \epsilon Y_1(\cdot), \tilde{Z}(\cdot) + \epsilon Z_1(\cdot))\).Then, it follows that
\begin{equation*}
	\begin{aligned}
		J_{\bm{\eta}}&(u^\epsilon(\cdot),\bm{\lambda})-J_{\bm{\eta}}(\tilde{u}_{\bm{\eta}}^{\bm{\lambda}}(\cdot),\bm{\lambda})\\
		=&\epsilon^2 \mathbb{E}\Big\{\langle GY_1(0),Y_1(0)\rangle+\int_0^T\langle Q(s)Y_1(s),Y_1(s)\rangle+\langle R(s)Z_1(s).Z_1(s)\rangle+\langle N(s)v(s),v(s)\rangle ds\Big\}\\
		&+\epsilon\mathbb{E}\Big\{ \langle G \tilde Y(0), Y_1(0)\rangle + \int_0^T[\langle Q(s) \tilde Y(s), Y_1(s)\rangle   +\langle R(s) \tilde Z(s), Z_1(s)\rangle+ \langle N(s)\tilde{u}_{\bm{\eta}}^{\bm{\lambda}}(s), v(s) \rangle   ] ds\\
		& + 	\int_0^T \langle   \lambda_1 (s), Y_1(s) \rangle ds 		+\int_0^T \langle  \lambda_2 (s), Z_1(s) \rangle ds+\int_0^T \langle \lambda_3(s),v(s)\rangle ds\Big\}.
	\end{aligned}
\end{equation*}

Based on the assumption (H2), we can conclude that
\begin{equation*}
	\begin{aligned}
		\mathbb{E}\Big\{\langle GY_1(0),Y_1(0)\rangle+\int_0^T\langle Q(s)Y_1(s),Y_1(s)\rangle+\langle R(s)Z_1(s).Z_1(s)\rangle+\langle N(s)v(s),v(s)\rangle ds\Big\}\geq 0,
	\end{aligned}
\end{equation*}
which indicates that
\begin{equation*}
	\begin{aligned}
		J_{\bm{\eta}}&(u^\epsilon(\cdot),\bm{\lambda})-J_{\bm{\eta}}(\tilde{u}_{\bm{\eta}}^{\bm{\lambda}}(\cdot),\bm{\lambda})\geq 0,\quad\forall \epsilon>0,\quad v(\cdot)\in\mathcal{U}[0,T].
	\end{aligned}	
\end{equation*}

Thus, the proof is complete.
\qed

At the end of this section, we focus on decoupling the fully coupled FBSDE (\ref{eq1}) using the invariant embedding technique. For \(s \in [0, T]\), we proceed as follows:
\begin{align}\label{yy}
	\tilde Y(s)=-\Sigma(s) \tilde X(s)-\phi(s),
\end{align}
where 
\begin{equation*}
	\left\{
	\begin{aligned}
		&	d \Sigma(s)=-\Sigma_1(s)ds-\Phi(s)dW(s),\\
		&	\Sigma(T)=0,
	\end{aligned}
	\right. 
\end{equation*}
and 
\begin{equation*}
	\left\{
	\begin{aligned}
		&d\phi(s)=-\phi_1(s)ds-\varphi(s)dW(s),\\
		&\phi(T)=-\xi,
	\end{aligned}
	\right.
\end{equation*}
with \(\Sigma_1(\cdot)\) and \(\phi_1(\cdot)\) being determined later. For simplicity,  we will dropped the dependence of the processes on \( s \) in the following text.

Applying It\^{o}'s formula,  we obtain 
\begin{align*}
	d\tilde Y=&\Sigma_1\tilde Xds+\Phi\tilde X dW(s)+\Sigma \big\{
	(A^\top \tilde X- Q\tilde Y
	- \lambda_1 ) d s +(C^\top \tilde X-R \tilde Z-\lambda_2  ) d W(s)
	\big\}\\
	&-\Phi(C^\top \tilde X- R \tilde Z-\lambda_2  ) ds+\phi_1ds+\varphi dW(s).
\end{align*}

Comparing above equation with the first equation of (\ref{eq1}), we get
\begin{equation}\label{YZ}
	\left\{
	\begin{aligned}
		0=&\Sigma_1\tilde X+\Sigma(A ^\top \tilde X- Q\tilde Y - \lambda_1 ) -\Phi(C^\top \tilde X
		- R \tilde Z-\lambda_2  )+\phi_1\\
		&-\{A \tilde Y+BN ^ {-1}B^\top \tilde X 
		-BN^ {-1}\lambda_3   +C \tilde  Z+\bar{A}\alpha+\bar{B} \gamma +\bar{C} \beta\},\\
		\tilde Z=&\Phi \tilde X  +\Sigma (C ^\top \tilde X- R \tilde Z-\lambda_2  )+\varphi.
	\end{aligned}
	\right.
\end{equation}

If  \(I_n +\Sigma R \) is invertible,\  we  have 
\begin{align}\label{Z}
	\tilde Z=\big(I_n+\Sigma R\big)^{-1}\Big\{\big(\Phi+\Sigma C^\top\big)\tilde X-\Sigma \lambda_2+\varphi \Big\}.
\end{align}

Now inserting (\ref{Z}) and \(\tilde Y=-\Sigma \tilde X-\phi \) into the first equation of (\ref{YZ}),  we have
\begin{align*}
	0=&\Sigma_1\tilde X +\Sigma[A ^\top \tilde X+Q(\Sigma \tilde X+\phi) - \lambda_1 ] -\Phi\{C^\top \tilde X-R(I_n+\Sigma R)^{-1}\\
	&[(\Phi+\Sigma C^\top)\tilde X-\Sigma\lambda_2+\varphi]-\lambda_2  \}+\phi_1+\{A (\Sigma \tilde X+\phi)-B N^ {-1}B^\top \tilde X \\
	& +B N^ {-1}\lambda_3  -C (I_n+\Sigma R)^{-1}[(\Phi+\Sigma C^\top)\tilde X-\Sigma\lambda_2+\varphi]-\bar{A}\alpha-\bar{B} \gamma -\bar{C} \beta\}.
\end{align*}
Namely 
\begin{align*}
	0=&\Big\{\Sigma_1+\Sigma A^\top+\Sigma  Q \Sigma -\Phi C ^\top+\Phi R (I_n+\Sigma R )^{-1}(\Phi+\Sigma C^\top)+A\Sigma-BN^{-1}B^\top\\
	&-C(I_n+\Sigma R)^{-1}(\Phi+\Sigma C^\top)\Big\}\tilde X\\
	&+\Sigma Q \phi -\Sigma \lambda_1 +\Phi [R (I_n +\Sigma R)^{-1}(-\Sigma \lambda_2 +\varphi )-\lambda_2 ]+\phi_1 +A \phi+ B N ^{-1}\lambda_3 \\
	&-C (I_n +\Sigma R )^{-1}(-\Sigma \lambda_2 +\varphi )-\bar A  \alpha -\bar B \gamma -\bar{C}\beta.
\end{align*}

Therefore, we can define \(\Sigma_1(\cdot)\) and \(\phi_1(\cdot)\) as follows:
\begin{align*}
	\Sigma_1=&-\Big(\Sigma A^\top+A\Sigma + \Sigma  Q \Sigma -BN^{-1}B^\top-\Phi C ^\top+\Phi R (I_n +\Sigma R )^{-1}(\Phi+\Sigma C^\top)\\
	&-C(I_n+\Sigma R)^{-1}(\Phi+\Sigma C^\top)\Big)
\end{align*}
and
\begin{align*}
	\phi_1=&- \Big((\Sigma Q+A) \phi  +(\Phi R -C)(I_n +\Sigma R)^{-1} \varphi -\Sigma\lambda_1 + B N ^{-1}\lambda_3 \\
	&+\big(C (I_n+\Sigma R )^{-1}\Sigma -\Phi R (I +\Sigma R)^{-1} ( \Sigma +I_n) \big)\lambda_2  -\bar A  \alpha -\bar B \gamma -\bar{C}\beta\Big).
\end{align*}

Moreover, 
\begin{align*}
	&- \Phi C ^\top +\Phi R (I_n +\Sigma R )^{-1}\Sigma C^\top  \\
	=& - \Phi (I _n+R \Sigma )^{-1} (I_n +R \Sigma  ) C ^\top   + \Phi  (I_n +R \Sigma )^{-1} R \Sigma C^\top \\
	=& - \Phi (I _n+R \Sigma )^{-1}  C ^\top,
\end{align*} 
Therefore,  
\begin{align*}
	\Sigma_1=&-\Big(\Sigma A^\top+A\Sigma + \Sigma  Q \Sigma -BN^{-1}B^\top+\Phi R (I_n+\Sigma R )^{-1}\Phi - \Phi (I _n+R \Sigma )^{-1}  C ^\top\\
	&-C(I_n+\Sigma R)^{-1}(\Phi+\Sigma C^\top)\Big).
\end{align*}
From the above, we can deduce the following BSDEs for \(\Sigma(\cdot)\) and \(\phi(\cdot)\): for \( s \in [0,T] \),
\begin{equation}\label{eq:Sigma}
	\left\{
	\begin{aligned}
		d\Sigma(s)=&\Big\{\Sigma A^\top+A\Sigma + \Sigma  Q \Sigma -BN^{-1}B^\top+\Phi R (I_n +\Sigma R )^{-1}\Phi - \Phi (I _n+R \Sigma )^{-1}  C ^\top\\
		&-C(I_n+\Sigma R)^{-1}(\Phi+\Sigma C^\top)\Big\}ds-\Phi(s)dW(s),\\
		\Sigma(T)=&0,
	\end{aligned}
	\right.
\end{equation}
and
\begin{equation}\label{eq:phi}
	\left\{
	\begin{aligned}
		d\phi(s)=&\Big\{(\Sigma Q+A) \phi  +(\Phi R -C)(I_n +\Sigma R)^{-1} \varphi -\Sigma \lambda_1 + B N ^{-1}\lambda_3 +\big(C (I_n+\Sigma R )^{-1}\Sigma\\
		& -\Phi R (I_n+\Sigma R)^{-1} ( \Sigma +I_n) \big)\lambda_2  -\bar A  \alpha -\bar B \gamma -\bar{C}\beta\}ds-\varphi(s)dW(s),\\
		\phi(T)=&-\xi.
	\end{aligned}
	\right.
\end{equation}

We now discuss the existence and uniqueness of the solution to (\ref{eq:Sigma}) and (\ref{eq:phi}), which is established by the following lemma.

\begin{lemma}\label{lem:sigmaphi}
	Suppose (H1) and (H2) hold. Then the stochastic Riccati equation (\ref{eq:Sigma}) admits a unique solution \(\big(\Sigma(\cdot),\Phi(\cdot)\big)\in L^{\infty,c}_\FF(\mathbb{S}^n)\times L^2_\FF(\mathbb{S}^n)\). 
	Moreover, the linear BSDE (\ref{eq:phi}) admits a unique solution \(\big(\phi(\cdot),\varphi(\cdot)\big)\in L^{2,c}_\FF(\RR^n)\times L^{\frac{1}{2}, 2}_\FF(\RR^n)\).
\end{lemma}
\begin{proof}
	The existence  of a solution to (\ref{eq:Sigma}), as well as the existence and uniqueness of a solution to (\ref{eq:phi}), are established by Sun and Wang in Theorem 5.3 and Theorem 5.5 respectively of \cite{sun2021linear}. Now, we present the proof of the uniqueness of the solution to equation (\ref{eq:Sigma}).
	
	Based on the Lemma \ref{lem:XYZ1}, the equation (\ref{yy}), and the uniqueness of \(\phi\) in (\ref{eq:phi}), we can derive the uniqueness of \(\Sigma\) in equation (\ref{eq:Sigma})  immediately.
	
	On the other hand, we define \(\tilde{\Sigma} = \Sigma - \bar{\Sigma}\) and \(\tilde{\Phi} = \Phi - \bar{\Phi}\), where \((\Sigma, \Phi)\) and \((\bar{\Sigma}, \bar{\Phi})\) are the solutions to equation (\ref{eq:Sigma}). By applying Itô’s formula to \(\tilde{\Sigma}^2\) and using the fact that \(\tilde{\Sigma}(s) = 0\) for all \(s \in [0, T]\), we obtain that \(\mathbb{E}\left[\int_0^T |\tilde{\Phi}(s)|^2 \, ds \right]=0\), which implies that \(\tilde{\Phi} = 0\) a.e. for all  \(s\in[0, T]\).
	
	Therefore, the proof is complete. 
\end{proof}

In the following, we fixed $(\Sigma(\cdot), \Phi(\cdot))$.
Based on above lemma, we can directly deduce the following conclusion, which provides a expression for the optimal control \(\tilde{u}_{\bm{\eta}}^{\bm{\lambda}}(\cdot)\).
\begin{lemma}\label{lem:deu}
	Let (H1) and (H2) hold. The control \(\tilde{u}_{\bm{\eta}}^{\bm{\lambda}}(\cdot)\) takes the following  form:\ for \(s \in [0,T]\),
	\begin{align}\label{eq:deu}
		\tilde u_{\bm{\eta}}^{\bm{\lambda}}(s)=N(s)^{-1}\{B(s)^\top \tilde X(s)-\lambda_3(s)\},
	\end{align}
	where  \(\tilde X(\cdot)\) is the unique solution of the following SDE: for \(\forall s\in[0,T]\),
	\begin{equation}\label{eq:tildeY} 
		\left\{
		\begin{aligned}
			d \tilde X(s)= &\; -\Big\{(A^\top+Q\Sigma) \tilde X+Q\phi- \lambda_1 \Big\} d s \\
			&-\Big\{\big(C^\top-R \big(I_n+\Sigma R\big)^{-1}\big(\Phi+\Sigma C^\top\big)\big)\tilde X\\
			&- R \big(I_n+\Sigma R\big)^{-1}\big(-\Sigma \lambda_2+\varphi\big)-\lambda_2  \Big\} d W(s),  \\
			\tilde X(0)= & \; - \big(I_n+G\Sigma (0) \big)^{-1}G \phi (0).
		\end{aligned}
		\right. \\
	\end{equation}
	where \(\big(\Sigma(\cdot), \Phi(\cdot)\big)\) is a solution of (\ref{eq:Sigma}), and \(\big(\phi(\cdot), \varphi(\cdot)\big)\) is the unique solution of (\ref{eq:phi}).
\end{lemma}

The following lemma provides the expression for the optimal control of Problem (a).

\begin{theorem}\label{th:uast}	
	Suppose that the Assumptions (H1) and (H2) hold, the optimal control \(u^\ast_{\bm{\eta}}(\cdot)\) can be expressed as the following  form:
	\begin{equation}\label{eq:uast}
		\begin{aligned}
			u^\ast_{\bm{\eta}}(s)=&N(s)^{-1}\{B(s)^\top  X^*(s)-\lambda^*_3(s)\},
		\end{aligned}
	\end{equation}
	where  \( X^*(\cdot)\) is the unique solution of the following SDE: for \(\forall s\in[0,T]\),
	\begin{equation}
		\left\{
		\begin{aligned}
			d  X^*(s)= &\; -\Big\{(A^\top+Q\Sigma) X^*+Q\phi^*- \lambda_1^* \Big\} d s \\
			&-\Big\{\big(C^\top-R \big(I_n+\Sigma R\big)^{-1}\big(\Phi+\Sigma C^\top\big)\big) X^*\\
			&- R \big(I_n+\Sigma R\big)^{-1}\big(-\Sigma \lambda^*_2+\varphi^*\big)-\lambda^*_2  \Big\} d W(s),  \\
			X^*(0)= & \; - \big(I_n+G\Sigma (0) \big)^{-1}G \phi ^*(0).
		\end{aligned}
		\right. \\
	\end{equation}
	where  \(\big(\phi^*(\cdot), \varphi^*(\cdot)\big)\) is the unique solution of (\ref{eq:phi}), obtained by replacing \(\lambda_2(\cdot)\) and \(\lambda_3(\cdot)\)  by \(\lambda^*_2(\cdot)\) and \(\lambda^*_3(\cdot)\) respectively,  with  \(\lambda^*_2(\cdot)\) and \(\lambda^*_3(\cdot)\) being an arbitrary solution of (\ref{eq:operators}). 
\end{theorem}

\section{The proof to Theorem \ref{main}}\label{sec6}

In this section, we give the proof of Theorem \ref{main}. We first prove the Lemma \ref{lem:conditions}.

\noindent\textit{Proof of Lemma \ref{lem:conditions}: }
Recalling equations (\ref{eq:uastoperator}), (\ref{eq:YZoperator}) and (\ref{eq:YZoperator1}), we can rewrite the cost functional \(\hat{J}_{\bm{\eta}}( u(\cdot))\) as
\begin{equation*}
	\begin{aligned}
		\hat{J}_{\bm{\eta}}(t,\xi;u^\ast_{\alpha,\beta,\gamma}(\cdot))&=\mathbb{E} \Big\{ \langle G Y(0), Y(0) \rangle 
		+ \int_0^T \Big[ \langle Q(s) Y(s), Y(s) \rangle + \langle \bar{Q}(s) \alpha(s), \alpha(s) \rangle \\
		&\quad\quad\quad+ \langle R(s) Z(s), Z(s) \rangle + \langle \bar{R}(s) \beta(s), \beta(s) \rangle + \langle N(s) u(s), u(s) \rangle\\
		&  \quad\quad\quad+\langle \bar{N}(s) \gamma(s), \gamma(s) \rangle \Big] ds \Big\}\\
		&= \langle(\cO_{31}^\ast G\cO_{31}+\cO_{01}^\ast Q\cO_{01}+\cO_{11}^\ast R\cO_{11}+\cO_{21}^\ast N\cO_{21})\xi,\xi\rangle_{\LL^2}\\
		&\quad+\langle (\cO_{32}^\ast G\cO_{32}+\cO_{02}^\ast Q\cO_{02}+\bar{Q}+\cO_{12}^\ast R\cO_{12}+\cO_{22}^\ast N\cO_{22})\alpha,\alpha\rangle_{\LL^2}\\
		&\quad+\langle (\cO_{33}^\ast G\cO_{33}+\cO_{03}^\ast Q\cO_{03}+\cO_{13}^\ast R\cO_{13}+\cO_{23}^\ast N\cO_{23}+\bar{R})\beta,\beta\rangle_{\LL^2}\\
		&\quad+\langle (\cO_{34}^\ast G\cO_{34}+\cO_{04}^\ast Q\cO_{04}+\cO_{14}^\ast R\cO_{14}+\cO_{24}^\ast N\cO_{24}+\bar{N})\gamma,\gamma\rangle_{\LL^2}\\
		&\quad+2\langle (\cO_{32}^\ast G\cO_{31}+\cO_{02}^\ast Q\cO_{01}+\cO_{12}^\ast R\cO_{11}+\cO_{22}^\ast N\cO_{21})\xi,\alpha\rangle_{\LL^2}\\
		&\quad+2\langle (\cO_{33}^\ast G\cO_{31}+\cO_{03}^\ast Q\cO_{01}+\cO_{13}^\ast R\cO_{11}+\cO_{23}^\ast N\cO_{21})\xi,\beta\rangle_{\LL^2}\\
		&\quad+2\langle (\cO_{34}^\ast G\cO_{31}+\cO_{04}^\ast Q\cO_{01}+\cO_{14}^\ast R\cO_{11}+\cO_{24}^\ast N\cO_{21})\xi,\gamma\rangle_{\LL^2}\\
		&\quad+2\langle (\cO_{33}^\ast G\cO_{32}+\cO_{03}^\ast Q\cO_{02}+\cO_{13}^\ast R\cO_{12}+\cO_{23}^\ast N\cO_{22})\alpha,\beta\rangle_{\LL^2}\\
		&\quad+2\langle (\cO_{34}^\ast G\cO_{32}+\cO_{04}^\ast Q\cO_{02}+\cO_{14}^\ast R\cO_{12}+\cO_{24}^\ast N\cO_{22})\alpha,\gamma\rangle_{\LL^2}\\
		&\quad+2\langle (\cO_{34}^\ast G\cO_{33}+\cO_{04}^\ast Q\cO_{03}+\cO_{14}^\ast R\cO_{13}+\cO_{24}^\ast N\cO_{23})\beta,\gamma\rangle_{\LL^2}.
	\end{aligned}
\end{equation*}

Therefore, by the first-order conditions, we  obtain that \(\alpha^\ast(\cdot)\), \(\beta^\ast(\cdot)\), and \(\gamma^\ast(\cdot)\) are the optimal solutions for Problem (b) if and only if 
\begin{equation*}
	\left\{
	\begin{aligned}
		&(\cO_{32}^\ast G\cO_{31}+\cO_{02}^\ast Q\cO_{01}+\cO_{12}^\ast R\cO_{11}+\cO_{22}^\ast N\cO_{21})\xi+(\cO_{32}^\ast G\cO_{32}+\cO_{02}^\ast Q\cO_{02}+\bar{Q}\\
		&\quad+\cO_{12}^\ast R\cO_{12}+\cO_{22}^\ast N\cO_{22})\alpha^*+ (\cO_{32}^\ast G\cO_{33}+\cO_{02}^\ast Q\cO_{03}+\cO_{12}^\ast R\cO_{13}+\cO_{22}^\ast N\cO_{23})\beta^*\\
		&\quad+(\cO_{32}^\ast G\cO_{34}+\cO_{02}^\ast Q\cO_{04}+\cO_{12}^\ast R\cO_{14}+\cO_{22}^\ast N\cO_{24})\gamma^*
		=0,\\
		&(\cO_{33}^\ast G\cO_{33}+\cO_{03}^\ast Q\cO_{03}+\cO_{13}^\ast R\cO_{13}+\cO_{23}^\ast N\cO_{23}+\bar{R})\beta^*+
		(\cO_{33}^\ast G\cO_{31}+\cO_{03}^\ast Q\cO_{01}\\
		&\quad+\cO_{13}^\ast R\cO_{11}+\cO_{23}^\ast N\cO_{21})\xi+(\cO_{33}^\ast G\cO_{32}+\cO_{03}^\ast Q\cO_{02}+\cO_{13}^\ast R\cO_{12}+\cO_{23}^\ast N\cO_{22})\alpha^*\\
		&\quad	+(\cO_{33}^\ast G\cO_{34}+\cO_{03}^\ast Q\cO_{04}+\cO_{13}^\ast R\cO_{14}+\cO_{23}^\ast N\cO_{24})\gamma^*=0,\\
		&(\cO_{34}^\ast G\cO_{34}+\cO_{04}^\ast Q\cO_{04}+\cO_{14}^\ast R\cO_{14}+\cO_{24}^\ast N\cO_{24}+\bar{N})\gamma^*
		+(\cO_{34}^\ast G\cO_{31}+\cO_{04}^\ast Q\cO_{01}\\
		&\quad+\cO_{14}^\ast R\cO_{11}+\cO_{24}^\ast N\cO_{21})\xi+(\cO_{34}^\ast G\cO_{32}+\cO_{04}^\ast Q\cO_{02}+\cO_{14}^\ast R\cO_{12}+\cO_{24}^\ast N\cO_{22})\alpha^*\\
		&\quad+ (\cO_{34}^\ast G\cO_{33}+\cO_{04}^\ast Q\cO_{03}+\cO_{14}^\ast R\cO_{13}+\cO_{24}^\ast N\cO_{23})\beta^*=0.
	\end{aligned}
	\right.
\end{equation*}
Furthermore, the above equation can be written in the form of the following matrix product
\begin{align*}
	\mathbf{I}_0\big(\mathcal{O}^\top\mathcal{T}\mathcal{O}+\mathcal{W}\big)(\xi,\bm{\eta}^\ast)^\top=(0,0,0,0)^\top,
\end{align*}
where \(\mathbf{I}_0= 
\begin{pmatrix}
	0&0&0&0\\
	0&1&0&0\\
	0&0&1&0\\
	0&0&0&1
\end{pmatrix}\), \(\mathcal{O}=
\begin{pmatrix}
	\mathcal{O}_{01}&\mathcal{O}_{02}&\mathcal{O}_{03}&\mathcal{O}_{04}\\
	\mathcal{O}_{11}&\mathcal{O}_{12}&\mathcal{O}_{13}&\mathcal{O}_{14}\\
	\mathcal{O}_{21}&\mathcal{O}_{22}&\mathcal{O}_{23}&\mathcal{O}_{24}\\
	\mathcal{O}_{31}&\mathcal{O}_{32}&\mathcal{O}_{33}&\mathcal{O}_{34}
\end{pmatrix}
\),
\(\mathcal{T}=
\begin{pmatrix}
	Q&0&0&0\\
	0&R&0&0\\
	0&0&N&0\\
	0&0&0&G
\end{pmatrix}\), and \(\mathcal{W}=
\begin{pmatrix}
	0&0&0&0\\
	0&\bar{Q}&0&0\\
	0&0&\bar{R}&0\\
	0&0&0&\bar{N}
\end{pmatrix}\).

Thus, we have now proved this lemma.	
\qed

Finally, we will present the proof of the main theorem in this paper.

\noindent\textit{Proof  of Theorem \ref{main}:}
It is evident that
\[
\inf_{u \in \mathcal{U}} J(u(\cdot)) = \inf_{\alpha(\cdot), \beta(\cdot), \gamma(\cdot) \in \mathbb{L}^2} \inf_{u(\cdot) \in \mathcal{U}} \left\{ J(u) : \mathbb{E}[Y^{u}(\cdot)] = \alpha(\cdot), \mathbb{E}[Z^u(\cdot)] = \beta(\cdot), \text{ and } \mathbb{E}[u(\cdot)] = \gamma(\cdot) \right\}.
\]
Therefore, it is natural to decompose Problem (MF-BSLQ) into Problem (a) and Problem (b).  Applying Lemmas \ref{lem:p11}, \ref{lem:XYZ1}, \ref{lem:uu} and \ref{lem:conditions}, then the Theorem \ref{main} can be easily deduced.
\qed

\end{document}